\documentclass{amsart}

\usepackage{amsmath, amssymb, amsthm}

\newtheorem{theorem}{Theorem}[section]
\newtheorem{lemma}[theorem]{Lemma}

\newtheorem{corollary}[theorem]{Corollary}
\newtheorem{claim}[theorem]{Claim}

\newtheorem{fact}[theorem]{Fact}
\newtheorem{question}[theorem]{Question}

\theoremstyle{definition}
\newtheorem{definition}[theorem]{Definition}
\newtheorem{remark}[theorem]{Remark}

\newcommand{\cf}{\mathrm{cf}}
\newcommand{\dom}{\mathrm{dom}}
\newcommand{\bb}{\mathbb}

\newcommand{\power}{\mathcal{P}}
\newcommand{\la}{\langle}
\newcommand{\ra}{\rangle}

\newcommand{\po}{\mathbb{P}}
\renewcommand{\P}{\mathbb{P}}

\newcommand{\cof}{\mathrm{cof}}

\DeclareMathOperator{\spc}{sc}
\DeclareMathOperator{\otp}{otp}

\begin{document}
\title{Diagonal supercompact Radin forcing}
\author{Omer Ben-Neria}
\address{Einstein Institute of Mathematics, Hebrew University of Jerusalem}
\email{omer.bn@mail.huji.ac.il}
\author{Chris Lambie-Hanson}
\address{Department of Mathematics and Applied Mathematics \\
Virginia Commonwealth University \\
Richmond, VA 23284 \\ United States}
\email{cblambiehanso@vcu.edu}
\urladdr{people.vcu.edu/~cblambiehanso}
\author{Spencer Unger}
\address{Einstein Institute of Mathematics, Hebrew University of Jerusalem}
\email{unger.spencer@mail.huji.ac.il}
\thanks{This research was conducted while the second author was a Lady Davis Postdoctoral Fellow. The
author would like to thank the Lady Davis Fellowship Trust and the Hebrew University of Jerusalem.
The third author was partially supported by NSF DMS-1700425.
We would also like to thank the anonymous referee, who
made a number of helpful corrections and suggestions that substantially improved
the paper.
}
\date{\today}


\begin{abstract}
	Motivated by the goal of constructing a model in which there are no
	$\kappa$-Aronszajn trees for any regular $\kappa>\aleph_1$, we produce a
	model with many singular cardinals where both the singular cardinals hypothesis
	and weak square fail.
\end{abstract}

\maketitle

\section{Introduction}

In this paper, we produce a model of ZFC with some global
behavior of the continuum function on singular cardinals and the failure of weak
square. Our method is as an extension of Sinapova's work
\cite{sinapovauncountable1}.  We define a diagonal supercompact Radin forcing
which adds a club subset to a cardinal $\kappa$ while forcing the failure of
the Singular Cardinals Hypothesis (SCH)
everywhere on the club and preserving the inaccessibility of $\kappa$. In the
forcing extension, weak square will necessarily hold at some successors of
singular cardinals below $\kappa$, but the set of these singular cardinals will
be sufficiently sparse that it can be made non-stationary by $\kappa$-distributive
forcing. We will thus obtain the following result.

\begin{theorem}\label{thm1} If there are a supercompact cardinal $\kappa$ and a
weakly inaccessible cardinal $\theta > \kappa$, then there is a forcing extension in
which $\kappa$ is inaccessible and there is a club $E \subseteq \kappa$ of
singular cardinals $\nu$ at which SCH and $\square^*_\nu$ both fail.
\end{theorem}

We are motivated by the question of whether in ZFC one can construct a
$\kappa$-Aronszajn tree for some $\kappa>\omega_1$.  The question is also open
if we ask for a \emph{special} $\kappa$-Aronszajn tree. Forcing provides a
possible path to a negative solution by showing that it is consistent with ZFC
that there are no $\kappa$-Aronsajn trees on any regular $\kappa>\omega_1$.  By
a theorem of Jensen \cite{jensen}, $\square_\mu^*$ is equivalent to the
existence of a special $\mu^+$-Aronszajn tree.  So our theorem is partial
progress towards a model with no special Aronszajn trees.

The non-existence of $\kappa$-Aronszajn trees (the tree property at $\kappa$)
and the non-existence of special $\kappa$-Aronszajn trees (failure of
$\square^*$) are reflection principles which are closely connected
with large cardinals.  For example, theorems of Erd\H{o}s and Tarski
\cite{erdostarski}, and Monk and Scott \cite{monkscott}, show that an
inaccessible cardinal is weakly compact if and only if it has the tree property.
Further, Mitchell and Silver \cite{mitchell} showed that the tree property at
$\aleph_2$ is consistent with ZFC if and only if the existence of a weakly
compact cardinal is.

Specker \cite{specker} showed that, if $\kappa^{<\kappa}=\kappa$, then there is a
special $\kappa^+$-Aronszajn tree.  This theorem places an important restriction
on models where there are no special Aronszajn trees.  From Specker's theorem, a
model with no special $\kappa$-Aronszajn trees for any $\kappa>\aleph_1$ must be
one in which GCH fails everywhere.  In particular GCH must fail at every
singular strong limit cardinal, a failure of SCH.
The consistency of the failure of SCH requires large cardinals
\cite{gitiksch}; a model in which GCH fails everywhere was first obtained by
Foreman and Woodin \cite{FW}.

There are many partial results towards constructing a model in which every
regular cardinal greater than $\aleph_1$ has the tree property.  There is a
bottom up approach where one attempts to force longer and longer initial
segments of the regular cardinals to have the tree property; see, for example
\cite{abraham,cf,neemanupto,ungernsl}.  We refer the reader to
\cite{ungerind2} for some analogous results on successive failures of weak
square.  Another aspect of the problem comes from the interaction between
cardinal arithmetic at singular strong limit cardinals $\mu$ and the tree
property at $\mu^+$.  In the 1980's Woodin asked whether the failure of SCH at
$\aleph_{\omega}$ is consistent with the tree property at $\aleph_{\omega+1}$.
More generally, one can consider whether this situation is consistent at some
larger singular cardinal.  An important result in this direction is due to Gitik
and Sharon \cite{gitiksharon}, who showed that, relative to the existence of a
supercompact cardinal, it is consistent that there is a singular cardinal
$\kappa$ of cofinality $\omega$ such that SCH fails at $\kappa$ and there are no
special $\kappa^+$-Aronszajn trees.  In fact they show a stronger assertion
$\left(\kappa^+ \notin I[\kappa^+]\right)$, which we will define later.  In the same paper,
they show that it is possible to make $\kappa$ into $\aleph_{\omega^2}$.
Cummings and Foreman \cite{cfdiag} showed that there is a PCF theoretic object
called a bad scale in the models of Gitik and Sharon, which implies that
$\kappa^+ \notin I[\kappa^+]$.

The key ingredient in Gitik and Sharon's argument was a new diagonal
supercompact Prikry forcing.  The basic idea is to start with supercompactness
measures $U_n$ on $\mathcal{P}_\kappa\left(\kappa^{+n}\right)$ for $n<\omega$ and use them to
define a Prikry forcing. This forcing adds a sequence $\langle x_n \mid n<\omega
\rangle$, where each $x_n$ is a typical point for $U_n$ and
$\bigcup_{n<\omega}x_n = \kappa^{+\omega}$.  The result is that
$\kappa^{+\omega}$ is collapsed to have size $\kappa$ and
$\kappa^{+\omega+1}$ becomes the new successor of $\kappa$.  The fact that
$\kappa^{+\omega+1} \notin I[\kappa^{+\omega+1}]$ in the ground model persists
to provide $\kappa^+ \notin I[\kappa^+]$ in the extension.  Moreover, if we start
with $2^\kappa = \kappa^{+\omega+2}$ in the ground model, then we get the
failure of SCH at $\kappa$ in the extension.

Variations of Gitik and Sharon's poset have been used to construct many related
models.  We list a few such results:

\begin{enumerate}
\item (Neeman \cite{neemantpsch}) From $\omega$-many supercompact cardinals, there is
a forcing extension in which there is a singular cardinal $\kappa$ of cofinality
$\omega$ such that SCH fails at $\kappa$ and $\kappa^+$ has the tree property.

\item (Sinapova \cite{sinapovauncountable1}) From a supercompact cardinal
$\kappa$, for any regular $\lambda<\kappa$, there is a forcing extension in
which $\kappa$ is a singular cardinal of cofinality $\lambda$, SCH fails at
$\kappa$ and $\kappa$ carries a bad scale (in particular $\kappa^+ \notin
I[\kappa^+]$ and there are no special $\kappa^+$-Aronszajn trees).

\item (Sinapova \cite{sinapovauncountable2}) From $\lambda$-many supercompact
cardinals $\langle \kappa_\alpha \mid \alpha < \lambda \rangle$ with $\lambda <
\kappa_0$ regular, there is a forcing extension in
which $\kappa_0$ is a singular cardinal of cofinality $\lambda$, SCH fails at
$\kappa_0$ and $\kappa_0^+$ has the tree property.

\item (Sinapova \cite{sinapovaomegasquared}) From $\omega$-many supercompact
cardinals, it is consistent that Neeman's result above holds with
$\kappa=\aleph_{\omega^2}$.
\end{enumerate}

Woodin's original question remains open; see \cite{sinapovaunger2} for the best
known partial result.  A theme in the above results is that questions about the
tree property are answered by first constructing a model where there are no
special $\kappa^+$-Aronszajn trees (or even $\kappa^+ \notin I[\kappa^+]$).  To
obtain the tree property, one needs to increase the large cardinal assumption
and to give a version of an argument of Magidor and Shelah \cite{magidorshelah},
who showed that the tree property holds at $\mu^+$ when $\mu$ is a singular
limit of supercompact cardinals.  The results of our paper are based on the
ideas from Sinapova's \cite{sinapovauncountable1}, but we expect that they will
generalize to give the tree property in the presence of stronger large cardinal
assumptions.

The paper is organized as follows.  In Section \ref{background} we give some
definitions and background material required for the main result.  In Section
\ref{mainposet} we describe the main forcing for Theorem \ref{thm1} and
prove some of its basic properties.
In Section \ref{arithmetic_section}, we show that the main forcing gives a model
with a club $C$ of cardinals where SCH fails.  In
Section \ref{weaksquare} we characterize which cardinals in this club $C$ have
weak square sequences and show that this set can be made non-stationary by
$\kappa$-distributive forcing, thus completing the proof of Theorem \ref{thm1}.
In Section \ref{conclusion} we make some concluding remarks and ask some open
questions.

\section{Background} \label{background}

In this section we will make the notions from the introduction precise and give
some further definitions that are relevant to the rest of the paper.

\begin{definition}
We say that $\nu$ \emph{has a weak square
sequence} ($\square^*_\nu$) if there is a sequence $\langle C_\gamma \mid \gamma
<\nu^+ \rangle$ such that
\begin{enumerate}
\item for all $\gamma < \nu^+$ limit, $C_\gamma \subset \power(\gamma)$ is
nonempty of size at most $\nu$ such that, for every $c \in C_\gamma$, $c \subset
\gamma$ is club in $\gamma$ with $\otp(c) \leq \nu$, and
\item for all $\beta < \gamma < \nu$, if $\beta$ is a limit point of
some $c \in C_\gamma$, then $c \cap \beta \in C_\beta$.
\end{enumerate}
\end{definition}

\begin{definition} Let $\vec{z} = \la z_\alpha \mid \alpha < \nu^+\ra$ be a
sequence of bounded subsets of $\nu^+$. We say that a limit ordinal $\gamma$ is
\emph{$\vec{z}$-approachable} if there is an unbounded set $A \subset \gamma$
with $\otp(A) = \cf(\gamma)$ such that, for every $\beta < \gamma$, $A \cap \beta
= z_\alpha$ for some $\alpha<\gamma$.  The \emph{approachability
ideal} $I[\nu^+]$ consists of all subsets $S \subset \nu^+$ for which there are
$\vec{z}$ as above and a club $C \subset \nu^+$ so that every $\gamma \in C \cap
S$ is $\vec{z}$-approachable.  \end{definition}

By arranging that $z_{\alpha+1}$ is the closure of $z_\alpha$ for each $\alpha <
\nu^+$, we may assume that for every $\vec{z}$-approachable point $\gamma$,
there is a witness $A \subset \gamma$ which is closed.

\subsection{Forcing preliminaries}\label{forcingprelim}

In this subsection, we describe the preparation of the ground model over which we will force
with our diagonal supercompact Radin forcing. Begin with a model $V_0$ in
which $\mathrm{GCH}$ holds and $\kappa < \theta$ are cardinals, with $\kappa$
supercompact. Force over $V_0$ with Laver's forcing \cite{laver} to make
the supercompactness of $\kappa$ indestructible under $\kappa$-directed closed
forcing, and then force over the resulting model to add $\theta$-many
Cohen subsets to $\kappa$. Call this final model $V$; it will be our ground
model for the remainder of the paper.

The following lemma holds as in \cite{sinapovauncountable1}.

\begin{lemma}
	For all $\alpha < \theta$, for all $\mathcal{X} \subseteq
	\mathcal{P}(\mathcal{P}_\kappa(\kappa^{+\alpha}))$, there are a normal, fine
	ultrafilter $U$ on $\mathcal{P}_\kappa(\kappa^{+\alpha})$ and functions
	$\langle f_\eta \mid \eta < \theta \rangle$ from $\kappa$ to $\kappa$ such
	that, letting $j:V \rightarrow M \cong \mathrm{Ult}(V, U)$, we have
	\begin{itemize}
		\item $\mathcal{X} \in M$;
		\item for all $\eta < \theta$, $j(f_\eta)(\kappa) = \eta$.
	\end{itemize}
\end{lemma}

Now, again as in \cite{sinapovauncountable1}, by recursion on $\alpha < \theta$,
we can construct a sequence of ultrafilters $\vec{U} = \langle U_\alpha \mid
\alpha < \theta \rangle$ and, for all $\alpha < \theta$, a sequence
$\langle f^\alpha_\eta \mid \eta < \theta \rangle$ such that the following hold.

\begin{itemize}
	\item For all $\alpha < \theta$, $U_\alpha$ is a normal, fine ultrafilter on
    $\mathcal{P}_\kappa(\kappa^{+\alpha})$. Let $j_\alpha : V \rightarrow M_\alpha
    \cong \mathrm{Ult}(V, U_\alpha)$ be the collapsed ultrapower map.
	\item For all $\alpha < \beta < \theta$, $U_\alpha \in M_\beta$.
	\item For all $\alpha < \theta$, $\kappa$ is $\kappa^{+\alpha}$-supercompact
	in $M_\alpha$.
	\item For all $\alpha, \eta < \theta$, we have $f^\alpha_\eta :
    \kappa \rightarrow \kappa$ and $j_\alpha(f^\alpha_\eta)(\kappa) = \eta$.
\end{itemize}

When we write that something happens for most (or for almost all) $x \in
\mathcal{P}_\kappa(\kappa^{+\alpha})$, we mean it happens for a
$U_\alpha$-measure one set. For $\alpha < \theta$, for most $x \in
\mathcal{P}_\kappa(\kappa^{+\alpha})$, $x \cap \kappa$ is an inaccessible
cardinal.  We will always work with such $x$ and will write $\kappa_x$ for $x \cap
\kappa$.  For $x, y \in \mathcal{P}_\kappa(\kappa^{+\alpha})$, $x \prec y$
denotes the statement that $x \subseteq y$ and $\mathrm{otp}(x) < \kappa_y$.

For $\alpha < \beta < \theta$, let $\bar{u}^\beta_\alpha$ be a function on
$\mathcal{P}_\kappa(\kappa^{+\beta})$ representing $U_\alpha$ in the ultrapower
by $U_\beta$. For most $x \in \mathcal{P}_\kappa(\kappa^{+\beta})$,
$\bar{u}^\beta_\alpha(x)$ is a measure on
$\mathcal{P}_{\kappa_x}(\kappa_x^{+f^\beta_\alpha(\kappa_x)})$. Also, for most
$x \in \mathcal{P}_{\kappa}(\kappa^{+\beta})$, $\mathrm{otp}(x \cap \kappa^{+\alpha}) =
\kappa_x^{+f^\beta_\alpha(\kappa_x)}$. For such $x$, $\bar{u}^\beta_\alpha(x)$ is
isomorphic to a measure $u^\beta_\alpha(x)$ on $\mathcal{P}_{\kappa_x}(x \cap
\kappa^{+\alpha})$ via the order-isomorphism between $\kappa_x^{+f^\beta_\alpha(\kappa_x)}$
and $x \cap \kappa^{+\alpha}$.

For $y \in \mathcal{P}_\kappa(\kappa^{+\beta})$, let
$Z^\beta_y = \{\alpha < \beta \mid \kappa^{+\alpha} \in y\}$. Note that, for
most $y \in \mathcal{P}_\kappa(\kappa^{+\beta})$, we have $Z^\beta_y =
y \cap \beta$, so the following results also hold with $y \cap \beta$ in place
of $Z^\beta_y$. We feel that $Z^\beta_y$ is the more natural set to consider
in the context of the forcing defined in Section \ref{mainposet}, so we
will use it instead.

\begin{lemma} \label{nicenessLemma}
	For most $y \in \mathcal{P}_\kappa(\kappa^{+\beta})$, the following hold.
	\begin{enumerate}
		\item $Z^\beta_y$ is ${<}\kappa_y$-closed.
		\item If $\cf(\beta) < \kappa$, then $\cf(\beta) < \kappa_y$ and $Z^\beta_y$ is unbounded in $\beta$.
		\item $\mathrm{otp}(Z^\beta_y) = f^\beta_\beta(\kappa_y)$ and, if $\beta$
      is a limit ordinal, then so is $f^\beta_\beta(\kappa_y)$. Also, if
      $cf(\beta) \geq \kappa$ then $\cf(f^\beta_\beta(\kappa_y)) \geq \kappa_y$.
		\item For all $\alpha \in Z^\beta_y$, $\mathrm{otp}(y \cap \kappa^{+\alpha}) =
      \kappa_y^{+f^\beta_\alpha(\kappa_y)} = \kappa_y^{+\mathrm{otp}(\alpha \cap Z^\beta_y)}$.
		\item $\kappa_y$ is $\kappa_y^{+f^\beta_\beta(\kappa_y)}$-supercompact.
		\item For all $\alpha \in Z^\beta_y$, $\bar{u}^\beta_\alpha(y)$ is a measure
      on $\mathcal{P}_{\kappa_y}(\kappa_y^{+f^\beta_\alpha(\kappa_y)})$.
		\item For all $\alpha_0 < \alpha_1$, both in $Z^\beta_y$, the function
      $x \mapsto \bar{u}^{\alpha_1}_{\alpha_0}(x)$ represents $\bar{u}^\beta_{\alpha_0}(y)$
      in the ultrapower by $u^\beta_{\alpha_1}(y)$.
	\end{enumerate}
\end{lemma}

\begin{proof}
  Let $j = j_\beta$. Recall that a set $A$ is in $U_\beta$ iff $j``\kappa^{+\beta} \in j(A)$.
  Note first that, defining $g:\mathcal{P}_\kappa(\kappa^{+\beta}) \rightarrow V$
  by $g(y) = Z^\beta_y$, we have $j(g)(j``\kappa^{+\beta}) = j``\beta$. Items (1)--(5) then follow easily.

	To show (6), let $j(\langle \bar{u}^\beta_\alpha \mid \alpha < \beta \rangle) = \langle
  \bar{v}^{j(\beta)}_\alpha \mid \alpha < j(\beta) \rangle$ and $j(\langle f^\beta_\alpha
  \mid \alpha < \beta \rangle) = \langle g^{j(\beta)}_\alpha \mid \alpha < j(\beta) \rangle$.
  It suffices to show that, in $M_\beta$, for all $\alpha \in j``\beta$,
  $\bar{v}^{j(\beta)}_\alpha(j``\kappa^{+\beta})$ is a measure on
  $\mathcal{P}_\kappa(\kappa^{+g^{j(\beta)}_\alpha(\kappa)})$.
  Let $\alpha \in j``\beta$, with, say, $\alpha = j(\xi)$. Then
  $\bar{v}^{j(\beta)}_\alpha(j``\kappa^{+\beta}) = j(\bar{u}^\beta_\xi)(j``\kappa^{+\beta}) = U_\xi$,
  which is a measure on $\mathcal{P}_\kappa(\kappa^{+\xi}) =
  \mathcal{P}_\kappa(\kappa^{+g^{j(\beta)}_\alpha(\kappa)})$.

	We finally show (7). Let $j(\langle \bar{u}^{\alpha_1}_{\alpha_0} \mid \alpha_0 < \alpha_1 \leq
  \beta \rangle) = \langle \bar{v}^{\alpha_1}_{\alpha_0} \mid \alpha_0 < \alpha_1 \leq j(\beta) \rangle$
  and $j(\langle u^\beta_\alpha \mid \alpha < \beta \rangle) = \langle v^{j(\beta)}_\alpha \mid \alpha <
  j(\beta) \rangle$. It suffices to show that, in $M_\beta$, for all $\alpha_0 < \alpha_1$, both in
  $j``\beta$, the function $x \mapsto \bar{v}^{\alpha_1}_{\alpha_0}(x)$ represents
  $\bar{v}^{j(\beta)}_{\alpha_0}(j``\kappa^{+\beta})$ in the ultrapower by
  $v^{j(\beta)}_{\alpha_1}(j``\kappa^{+\beta})$. Fix $\alpha_0 < \alpha_1$ in $j``\beta$, with
  $\alpha_0 = j(\xi_0)$ and $\alpha_1 = j(\xi_1)$. Note that
	$\hat{U}_{\xi_1} := v^{j(\beta)}_{\alpha_1}(j``\kappa^{+\beta})$
  is a measure on $\mathcal{P}_\kappa(j``\kappa^{+\xi_1})$ that collapses to $U_{\xi_1}$.  Also note that $\bar{v}^{j(\beta)}_{\alpha_0}(j``\kappa^{+\beta}) = U_{\xi_0}$.
  Thus, we must show that the function $x \mapsto \bar{v}^{\alpha_1}_{\alpha_0}(x)$ represents $U_{\xi_0}$
  in the ultrapower by $\hat{U}_{\xi_1}$.

	Fix $x \in \mathcal{P}_\kappa(j``\kappa^{+\xi_1})$. There
  is $\bar{x} \in \mathcal{P}_\kappa(\kappa^{+\beta})$ such that $x = j(\bar{x})$.
  Then $\bar{v}^{\alpha_1}_{\alpha_0}(x) = j(\bar{u}^{\xi_1}_{\xi_0}(\bar{x}))$.
  For most $\bar{x} \in \mathcal{P}_\kappa(\kappa^{+\xi_1})$,
  $\bar{u}^{\xi_1}_{\xi_0}(\bar{x})$ is a measure on
  $\mathcal{P}_{\kappa_{\bar{x}}}(\kappa_{\bar{x}}^{+f^{\xi_1}_{\xi_0}(\kappa_{\bar{x}})})$,
  and this is fixed by $j$. Thus, for most $x \in \mathcal{P}_\kappa(j``\kappa^{+\xi_1})$,
  $\bar{v}^{\alpha_1}_{\alpha_0}(x) = \bar{u}^{\xi_1}_{\xi_0}(\bar{x})$. Therefore, since $\hat{U}_{\xi_1}$
  collapses to $U_{\xi_1}$, $x \mapsto \bar{v}^{\alpha_1}_{\alpha_0}(x)$ represents the same thing
  in the ultrapower by $\hat{U}_{\xi_1}$ as $\bar{x} \mapsto \bar{u}^{\xi_1}_{\xi_0}(x)$ represents in
  the ultrapower by $U_{\xi_1}$, which is $U_{\xi_0}$. This is true in $V$ and, since $M_\beta$ is
  sufficiently closed, it is true in $M_\beta$ as well.
\end{proof}

\begin{lemma} \label{addLemma1}
	Suppose $\beta < \theta$ and, for all $\alpha \leq \beta$, $A_\alpha \in U_\alpha$. Let $A^*$ be the
  set of all $y \in A_\beta$ such that, for all $\alpha \in Z^\beta_y$, $\{x \in A_\alpha \mid
  x \prec y \} \in u^\beta_\alpha(y)$. Then $A^* \in U_\beta$.
\end{lemma}

\begin{proof}
	Let $j = j_\beta$. It suffices to show that $j``\kappa^{+\beta} \in j(A^*)$, i.e. for all
  $j(\alpha) \in j``\beta$, $\{x \in j(A_\alpha) \mid x \prec j``\kappa^{+\beta} \} \in \hat{U}_\alpha$,
  where $\hat{U}_\alpha$ is the isomorphic copy of $U_\alpha$ living on $\mathcal{P}_\kappa(j``\kappa^{+\alpha})$.
  Fix such a $j(\alpha)$. Let $X = \{x \in j(A_\alpha) \mid x \prec j``\kappa^{+\beta} \}$,
  and note that $X = j``A_\alpha \in \hat{U}_\alpha$.
\end{proof}

\begin{lemma} \label{addLemma2}
	Suppose $\gamma < \theta$, $z \in \mathcal{P}_\kappa(\kappa^{+\gamma})$, and $z$ satisfies all of the
  statements in Lemma \ref{nicenessLemma}. Suppose that, for all $\alpha \in Z^\gamma_z$,
  $A_\alpha \in u^\gamma_\alpha(z)$. Fix $\beta \in Z^\gamma_z$, and let $A^*$ be the set of
  $y \in A_\beta$ such that, for all $\alpha \in Z^\beta_y$, $\{x \in A_\alpha \mid x \prec y \} \in
  u^\beta_\alpha(y)$. Then $A^* \in u^\gamma_\beta(z)$.
\end{lemma}

\begin{proof}
  For each $\alpha \in Z^\gamma_z$, let $\bar{A}_\alpha$ be the collapsed version of $A_\alpha$, so
  $\bar{A}_\alpha \in \bar{u}^\gamma_\alpha(z)$. Recall that $u^\gamma_\beta(z)$ is a measure on
  $\mathcal{P}_{\kappa_z}(z \cap \kappa^{+\beta})$. Let $k:V \rightarrow N \cong \mathrm{Ult}(V, u^\gamma_\beta(z))$
  be the ultrapower map. By (6) of Lemma \ref{nicenessLemma}, for all $\alpha \in Z^\gamma_z \cap \beta$,
  the map $y \mapsto \bar{u}^\beta_\alpha(y)$ represents $\bar{u}^\gamma_\alpha(z)$ in the ultrapower.
  Note also that the map $y \mapsto Z^\beta_y$ represents $\{\eta < k(\beta) \mid k(\kappa)^{+\eta}
  \in k``(z \cap \kappa^{+\beta})\} = k``(Z^\gamma_z \cap \beta)$. To prove the lemma, it suffices to show
  that $k``(z \cap \kappa^{+\beta}) \in k(A^*)$, i.e. for all $\alpha \in Z^\gamma_z \cap \beta$,
  $\{x \in k(A_\alpha) \mid x \prec k``(z \cap \kappa^{+\beta})\}$ is in the measure represented by
  the map $y \mapsto u^\beta_\alpha(y)$. Call this measure $w$ and note that it is isomorphic to
  $\bar{u}^\gamma_\alpha(z)$. Also note that $\{x \in k(A_\alpha) \mid x \prec k``(z \cap \kappa^{+\beta})\} =
  k``(A_\alpha)$, which collapses to $\bar{A}_\alpha \in \bar{u}^\gamma_\alpha(z)$. Thus, $\{x \in k(A_\alpha)
  \mid x \prec k``(z \cap \kappa^{+\beta})\} \in w$, completing the proof of the lemma.
\end{proof}

\section{The main forcing} \label{mainposet}

For $\beta < \theta$, let $X_\beta$ be the set of $y \in
\mathcal{P}_\kappa(\kappa^{+\beta})$ satisfying all of the statements in Lemma
\ref{nicenessLemma}. Fix $\eta < \theta$.  We define a forcing notion,
$\bb{P}_{\vec{U}, \eta}$. Conditions of $\bb{P}_{\vec{U}, \eta}$ are pairs $(a,
A)$ satisfying the following requirements.

\begin{enumerate}
	\item{$a$ and $A$ are functions, $\dom(a)$ is a finite subset of $\theta \setminus \eta$, and $\dom(A) = \theta \setminus (\dom(a) \cup \eta)$.}
	\item{For all $\beta \in \dom(a)$, $a(\beta) \in X_\beta$.}
	\item{For all $\alpha < \beta$, both in $\dom(a)$, $a(\alpha) \prec a(\beta)$ and $\alpha \in Z^\beta_{a(\beta)}$.}
	\item{For all $\alpha \in \theta \setminus (\max(\dom(a))+1)$ (or, if $\dom(a) = \emptyset$, for all $\alpha \in \dom(A)$), $A(\alpha) \in U_\alpha$.}
	\item{For all $\alpha \in \dom(A) \cap \max(\dom(a))$, if $\beta = \min(\dom(a) \setminus \alpha)$, then $A(\alpha) \in u^\beta_\alpha(a(\beta))$ if $\alpha \in Z^\beta_{a(\beta)}$ and $A(\alpha) = \emptyset$ if $\alpha \not\in Z^\beta_{a(\beta)}$.}
	\item{For all $\beta \in \dom(A)$ such that $A(\beta) \not= \emptyset$ and $\dom(a) \cap \beta \not= \emptyset$, if $\alpha = \max(\dom(a) \cap \beta)$, then for all $y \in A(\beta)$, $a(\alpha) \prec y$ and $\alpha \in Z^\beta_y$.}
\end{enumerate}

If $(a, A), (b, B) \in \bb{P}_{\vec{U}, \eta}$, then $(b, B) \leq (a, A)$ iff the following requirements hold.
\begin{enumerate}
	\item{$b \supseteq a$.}
	\item{For all $\alpha \in \dom(b) \setminus \dom(a)$, $b(\alpha) \in A(\alpha)$.}
	\item{For all $\alpha \in \dom(B)$, $B(\alpha) \subseteq A(\alpha)$.}
\end{enumerate}

$(b,B) \leq^* (a,A)$ if $(b,B) \leq (a,A)$ and $b = a$. In this case, $(b,B)$ is called a
\emph{direct extension} of $(a,A)$.

\begin{remark}
	In our arguments, for notational simplicity we will typically assume that $\eta = 0$
  and then denote $\bb{P}_{\vec{U}, \eta}$ as $\bb{P}_{\vec{U}}$. Everything proved
  about $\bb{P}_{\vec{U}}$ can be proved for a general $\bb{P}_{\vec{U}, \eta}$ in
  the same way by making the obvious changes. The reason we introduce the more general
  forcing is to be able to properly state the Factorization Lemma (\ref{factorizationLemma}).
\end{remark}

In what follows, let $\bb{P}$ denote $\bb{P}_{\vec{U}}$. For any condition
$p = (a, A) \in \bb{P}$, we often denote $(a, A)$ as $(a^p, A^p)$ and let
$\gamma^p = \max(\dom(a^p))$. We refer to $a^p$ as the \emph{stem} of $p$.
Note that, if $p, q \in \bb{P}$ and $a^p = a^q$, then $p$ and $q$ are compatible.
If $a$ is a non-empty stem, then let $\gamma^a$ denote $\max(\dom(a))$, and let
$a^- = a \restriction \gamma^a$. Suppose $a$ is a stem, $\alpha < \theta$, and
$x \in X_\alpha$. Suppose moreover that either $a$ is empty or $\gamma^a < \alpha$,
$a(\gamma^a) \prec x$, and $\gamma^a \in Z^\alpha_x$. Then $a ^\frown (\alpha, x)$
is a stem and $(a ^\frown (\alpha, x))^- = a$. If $p \in \bb{P}$ and $b$ is a stem,
then $b$ is \emph{possible} for $p$ if there is $q \leq p$ with $a^q = b$. If
$p \in \bb{P}$ and $b$ is possible for $p$, then $p \downarrow b$ denotes
the maximal $q$
such that $q \leq p$ and $a^q = b$. Such a $q$ always exists.

\begin{lemma} \label{addLemma3}
	Suppose that $(a, A) \in \bb{P}$, $\beta \in \dom(A)$, and $A(\beta) \not= \emptyset$.
  Then there is $(b, B) \leq (a, A)$ such that $\beta \in \dom(b)$.
\end{lemma}

\begin{proof}
	Let $\gamma := \min(\dom(a) \setminus \beta)$ if $\beta < \max(\dom(a))$,
	and let $\gamma := \theta$ otherwise. Let $\alpha_0 :=
	\max(\dom(a) \cap \beta)$ if $\dom(a) \cap \beta \neq \emptyset$ and
	$\alpha_0 = -1$ otherwise. By Lemma \ref{addLemma1} (if
	$\gamma = \theta$) or Lemma \ref{addLemma2} (if $\gamma < \theta$), we can find
	$y \in A(\beta)$ such that
	\begin{itemize}
		\item if $\dom(a) \cap \beta \neq \emptyset$, then $\alpha_0 \in Z^\beta_y$ and $a(\alpha_0) \prec y$;
		\item for all $\alpha \in Z^\beta_y \setminus (\alpha_0 + 1)$, we have
		$\{x \in A(\alpha) \mid x \prec y\} \in u^\beta_\alpha(y)$.
	\end{itemize}
	Now define a condition $(b, B)$ as follows. Let $\dom(b) = \dom(a) \cup \{\beta\}$,
	$b \restriction \dom(a) = a$, and $b(\beta) = y$. Let $\dom(B) = \theta \setminus
	\dom(b)$. For all $\alpha \in (\alpha_0, \beta)$, let $B(\alpha) =
	\{x \in A(\alpha) \mid x \prec y\}$. For $\alpha \in (\beta, \gamma)$, let
	$B(\alpha) = \{x \in A(\alpha) \mid y \prec x\}$. For all other
	$\alpha \in \dom(B)$, let $B(\alpha) = A(\alpha)$. It is easily verified that
	$(b, B)$ is a condition in $\bb{P}$ extending $(a, A)$.
\end{proof}

\begin{definition}
Suppose that $G$ is a $\bb{P}$-generic over $V$.
Let $C^{\spc}_G$ ($\spc$ for supercompact) be the set of all points $x = a(\beta)$ where $\beta \in \dom(a)$ for some $p = (a,A)$ in the generic filter $G$, and
let $C_G = \{ \kappa_x \mid x \in C^{\spc}_G\}$ be the generic Radin club.
\end{definition}

\begin{lemma} \label{genericClubLemma}
  $C_G$ is club in $\kappa$ and the assignment $x \mapsto \kappa_x = x \cap \kappa$ is an increasing bijection from $C^{\spc}_G$ to $C_G$.
\end{lemma}

\begin{proof}
	Straightforward by Lemma \ref{addLemma3} and genericity.
\end{proof}

\begin{lemma} \label{diagonal_intersection_lemma}(Diagonal Intersection Lemma)
	Suppose that
	\begin{itemize}
		\item $\beta < \theta$;
		\item $S$ is a set of stems with $\gamma^a < \beta$ for all $a \in S$;
		\item for each $a \in S$, we are given a set $Y_a \in U_\beta$.
	\end{itemize}
	Let $Z$ be the set of $y \in X_\beta$ such that, for all $a \in S$,
	if $\gamma^a \in Z^\beta_y$ and
	$a(\gamma^a) \prec y$, then $y \in Y_a$. Then $Z \in U_\beta$.
\end{lemma}

\begin{proof}
	Let $j = j_\beta$. It suffices to show that $j``\kappa^{+\beta} \in
	j(Z)$. Notice that, if $a \in j(S)$ is such that
	$\max(\dom(a)) \in j``\beta$ and $a(\max(\dom(a))) \prec j``\beta$, then
	there is $\bar{a} \in S$ such that $j(\bar{a}) = a$. But then
	$Y_{\bar{a}} \in U_\beta$, so $j``\kappa^{+\beta} \in j(Y_{\bar{a}})$.
	It follows that $j``\kappa^{+\beta} \in j(Z)$, as desired.
\end{proof}

Suppose that $\beta < \theta$ and $y \in X_\beta$. Let $\vec{U}_y = \langle \bar{u}^\beta_\alpha(y)
\mid \alpha \in Z^\beta_y \rangle$. For $\xi < f^\beta_\beta(\kappa_y)$, let $\alpha_\xi \in Z^\beta_y$
be such that $\mathrm{otp}(\alpha \cap Z^\beta_y) = \xi$. Then $\bar{u}^\beta_{\alpha_\xi}(y)$ is a
measure on $\mathcal{P}_{\kappa_y}(\kappa_y^{+\xi})$. Let $V_\xi = \bar{u}^\beta_{\alpha_\xi}(y)$.
Then $\vec{U}_y = \langle V_\xi \mid \xi < f^\beta_\beta(y) \rangle$, and we can define
$\bb{P}_{\vec{U}_y}$ as above.

If $p \in \bb{P}$, then $\bb{P}/p = \{q \in \bb{P} \mid q \leq p \}$.

\begin{lemma} \label{factorizationLemma}(Factorization Lemma)
	Let $p = (a, A) \in \bb{P}$. Suppose that $a \not= \emptyset$, $\gamma = \gamma^a$,
  and $y = a(\gamma)$. Then there is $p' \in \bb{P}_{\vec{U}_y}$ such that
  $\bb{P}/p \cong \bb{P}_{\vec{U}_y} / p' \times \bb{P}_{\vec{U}, \gamma + 1} /
  (\emptyset, A \restriction (\gamma, \theta))$.
\end{lemma}

\begin{proof}
	Let $\pi: y \rightarrow \mathrm{otp}(y)$ be the unique order-preserving bijection. Define
  $p' = (a', A') \in \bb{P}_{\vec{U}_y}$ as follows. For $\xi < f^\gamma_\gamma(y)$,
  let $\alpha_\xi \in Z^\gamma_y$ be such that $\mathrm{otp}(\alpha_\xi \cap Z^\gamma_y) = \xi$.
  Let $\dom(a') = \{\xi < f^\gamma_\gamma(y) \mid \alpha_\xi \in \dom(a)\}$ and, for $\xi \in \dom(a')$,
  let $a'(\xi) = \pi``a(\alpha_\xi)$. Then $\dom(A') = f^\gamma_\gamma(y) \setminus \dom(a')$.
  If $\xi \in \dom(A')$, let $A'(\xi) = \{\pi``x \mid x \in A(\alpha_\xi)\}$. It is straightforward
  to verify that $p'$ thus defined is in $\bb{P}_{\vec{U}_y}$ and that $\bb{P}/p \cong
  \bb{P}_{\vec{U}_y} / p' \times \bb{P}_{\vec{U}, \gamma + 1} / (\emptyset, A \restriction (\gamma, \theta))$.
\end{proof}

By repeatedly applying the Factorization Lemma, standard arguments (see, e.g.
\cite{gitikhandbook}) allow us to assume we are working below a condition of the form
$(\emptyset, A)$ when proving the following lemmas about $\bb{P}$.

\begin{lemma} \label{prikryLemma}
	$(\bb{P}, \leq, \leq^*)$ satisfies the Prikry property, i.e., if $\varphi$ is a statement
  in the forcing language and $p \in \bb{P}$, then there is $q \leq^* p$ such that $q \parallel \varphi$.
\end{lemma}

\begin{proof}
	The proofs of this and the next few lemmas are similar to those for the
  classical Radin forcing, which can be found in \cite{gitikhandbook}. Fix
  $\varphi$ in the forcing language and $p \in \bb{P}$. By the Factorization
  Lemma (\ref{factorizationLemma}), we may assume that $p = (\emptyset, A)$ for some $A$. Let $a$ be a
  stem possible for $p$, and let $\alpha \in \theta \setminus (\gamma^a + 1)$.
  Let $Y_{a, \alpha} = \{x \in A(\alpha) \mid a(\gamma^a) \prec x$ and $\gamma^a
  \in Z^\alpha_x \}$. Note that $Y_{a, \alpha} \in U_\alpha$. Let $Y^0_{a, \alpha}
  = \{x \in Y_{a, \alpha} \mid$ for some $B$, $(a^\frown (\alpha, x), B) \Vdash
  \varphi \}$, $Y^1_{a, \alpha} = \{x \in Y_{a, \alpha} \mid$ for some $B$,
  $(a^\frown (\alpha, x), B) \Vdash \neg \varphi \}$, and $Y^2_{a, \alpha} =
  Y_{a, \alpha} \setminus (Y^0_{a, \alpha} \cup Y^1_{a, \alpha})$. Fix $i(a,
  \alpha) < 3$ such that $Y^{i(a, \alpha)}_{a, \alpha} \in U_\alpha$, and let
  $Y^*_{a, \alpha} = Y^{i(a, \alpha)}_{a, \alpha}$.

	For $\alpha < \theta$, let $B(\alpha)$ be the set of $x \in A(\alpha)$ such that,
  for every stem $a$ possible for $p$ such that $a(\gamma^a) \prec x$ and $\gamma^a
  \in Z^\alpha_x$, $x \in Y^*_{a, \alpha}$. By the Diagonal Intersection Lemma
	(\ref{diagonal_intersection_lemma}), we have $B(\alpha) \in U_\alpha$.
	Thus, $(\emptyset, B) \in \bb{P}$ and $(\emptyset, B) \leq^* p$.

	Suppose for sake of contradiction that no direct extension of $(\emptyset, B)$ decides $\varphi$.
  Find $(a, B^*) \leq (\emptyset, B)$ deciding $\varphi$ with $|a|$ minimal.
  Without loss of generality, suppose that $(a, B^*) \Vdash \varphi$. Because of our assumption
  that no direct extension of $(\emptyset, B)$ decides $\varphi$, $a$ is non-empty.
  Let $b = a^-$ and $\gamma = \gamma^a$. By our construction of $B$, we have $a(\gamma)
  \in Y^0_{b, \gamma}$, and, for any $x \in B(\gamma)$ such that $b^\frown (\gamma, x)$ is a stem,
  there is $\hat{B}_x$ such that $(b^\frown (\gamma, x), \hat{B}_x) \Vdash \varphi$.
  Let $p^* = (\emptyset, B) \downarrow b = (b, B^{**})$. We will find a direct extension $(b, F)$ of $p^*$
  forcing $\varphi$, thus contradicting the minimality of $|a|$.

	We first define $F \restriction \gamma^b$ (if $b = \emptyset$, then there is nothing to do here).
  Since there are fewer than $\kappa$-many possibilities for $F \restriction \gamma^b$ and $U_\gamma$ is
  $\kappa$-complete, we may fix a function $F^*$ on $\gamma^b \setminus \dom(b)$ such that
  $B_0(\gamma) := \{x \in B(\gamma) \mid b^\frown (\gamma, x)$ is a stem and $\hat{B}_x \restriction
  \gamma^b = F^*\} \in U_\gamma$. Then, for all $\alpha \in \gamma^b \setminus \dom(b)$, let $F(\alpha) =
  F^*(\alpha) \cap B^{**}(\alpha)$. We next define $F$ on the interval $(\gamma^b, \gamma)$ (or on all of
  $\gamma$, if $b = \emptyset$). If $\alpha \in (\gamma^b, \gamma)$, $x \in B_0(\gamma)$, and
  $\alpha \in Z^\gamma_x$, note that $\hat{B}_x(\alpha) \in u^\gamma_\alpha(x)$. Let $\bar{B}_x(\alpha)$
  be the collapsed version of $\hat{B}_x(\alpha)$. Then $\bar{B}_x(\alpha) \in \bar{u}^\gamma_\alpha(x)$.
  Let $F^*(\alpha)$ be the set in $U_\alpha$ represented by the function $x \mapsto \bar{B}_x(\alpha)$
  in the ultrapower by $U_\gamma$, and let $F(\alpha) = F^*(\alpha) \cap B^{**}(\alpha)$. Let $F(\gamma)$
  be the set of $x \in B_0(\gamma) \cap B^{**}(\gamma)$ such that, for all $\alpha \in Z^\gamma_x \setminus
  (\gamma^b+1)$, $\{y \in F^*(\alpha) \mid y \prec x\} = \hat{B}_x(\alpha)$. We claim that $F(\gamma) \in U_\gamma$.
  To see this, let $j = j_\gamma$. Note that the function $x \mapsto \hat{B}_x(\alpha)$ represents
  $\{j``y \mid y \in F^*(\alpha)\}$, which is equal to $\{z \in j(F^*(\alpha)) \mid z \prec j``\kappa^{+\gamma} \}$.
  Thus, $j``\kappa^{+\gamma} \in j(F(\gamma))$, so $F(\gamma) \in U_\gamma$. We finally define $F$ on
  $(\gamma, \theta)$. If $\alpha \in (\gamma, \theta)$, let $F(\alpha)$ be the set of $y \in B^{**}(\alpha)$
  such that $\gamma \in Z^\alpha_y$ and, for all $x \in F(\gamma)$ such that $x \prec y$, $y \in \hat{B}_x(\alpha)$.
  Then $F(\alpha) \in U_\alpha$. Notice that, by our construction, if $(c, H)
  \leq (b, F)$ and $\gamma \in \dom(c)$, then $(c, H) \leq (b ^\frown (\gamma, c(\gamma)), \hat{B}_{c(\gamma)})$.

	Now suppose for sake of contradiction that $(b, F) \not\Vdash \varphi$. Find $(c, H) \leq (b, F)$
  such that $(c, H) \Vdash \neg\varphi$. If $\gamma \in \dom(c)$, then $(c, H) \leq (b ^\frown
  (\gamma, c(\gamma)), \hat{B}_{c(\gamma)}) \Vdash \varphi$, which is a contradiction. Thus,
  suppose $\gamma \not\in \dom(c)$. By our choice of $F(\alpha)$ for $\alpha \in (\gamma, \theta)$
  (namely, our requirement that $\gamma \in Z^\alpha_y$ for all $y \in F(\alpha)$), it must be the
  case that $H(\gamma) \not= \emptyset$. But then $(c, H)$ can be extended further to a condition
  $(c', H')$ such that $\gamma \in \dom(c')$, and this again gives a contradiction.
\end{proof}

The following definitions will play a crucial role in the proof that, if
$\theta$ is weakly inaccessible, then $\kappa$ remains strongly inaccessible in
the extension by $\bb{P}$.

\begin{definition}
	Let $n < \omega$. A tree $T \subseteq [\bigcup_{\alpha < \theta}(\{\alpha\} \times
  X_\alpha)]^{\leq n}$ is \emph{$\vec{U}$-fat} if the following conditions hold.
	\begin{enumerate}
		\item{For all $\langle (\alpha_i, x_i) \mid i \leq k \rangle \in T$ and all $i_0 < i_1
		\leq k$, we have $\alpha_{i_0} \in Z^{\alpha_{i_1}}_{x_{i_1}}$ and
		$x_{i_0} \prec x_{i_1}$.}
		\item{For all $t \in T$ with $\mathrm{lh}(t) < n$, there is $\alpha_t < \theta$ such that:
		\begin{enumerate}
			\item{for all $(\beta, y)$ such that $t ^\frown (\beta, y) \in T$, $\beta = \alpha_t$;}
			\item{$\{x \mid t^\frown (\alpha_t, x) \in T\} \in U_{\alpha_t}$.}
		\end{enumerate}}
	\end{enumerate}
\end{definition}

If $T$ is as in the previous definition, then $n$ is said to be the \emph{height} of $T$.

\begin{definition}
	Suppose that $T$ is a fat tree, $\alpha < \theta$, and $x \in X_\alpha$. $T$ is
	\emph{$\vec{U}$-fat above $(\alpha, x)$}
  if, for all $\langle (\alpha_i, x_i) \mid i \leq k \rangle \in T$ and all $i \leq k$, we
	have $\alpha \in Z^{\alpha_i}_{x_i}$, and $x \prec x_i$.
\end{definition}

\begin{definition}
	Suppose that $\gamma < \theta$ and $z \in X_\gamma$. A tree $T \subseteq [\bigcup_{\alpha < \theta}(\{\alpha\}
  \times X_\alpha)]^{\leq n}$ is
	\emph{$\vec{U}$-fat below $(\gamma, z)$} if the following conditions hold.
	\begin{enumerate}
		\item{For all $\langle (\alpha_i, x_i) \mid i \leq k \rangle \in T$ and all $i \leq k$,
      we have $\alpha_i \in Z^\gamma_z$ and $x_i \prec z$.}
		\item{For all $\langle (\alpha_i, x_i) \mid i \leq k \rangle \in T$ and all $i_0 < i_1 \leq k$,
        we have $\alpha_{i_0} \in Z^{\alpha_{i_1}}_{x_{i_1}}$ and $x_{i_0} \prec x_{i_1}$.}
		\item{For all $t \in T$ with $\mathrm{lh}(t) < n$, there is $\alpha_t \in Z^\gamma_z$ such that:
		\begin{enumerate}
			\item{for all $(\beta, y)$ such that $t ^\frown (\beta, y) \in T$, $\beta = \alpha_t$;}
			\item{$\{x \mid t^\frown (\alpha_t, x) \in T\} \in u^\gamma_{\alpha_t}(z)$.}
		\end{enumerate}}
	\end{enumerate}
\end{definition}

Notice that, if $T$ is $\vec{U}$-fat below $(\gamma, z)$, then it is isomorphic to
a $\vec{U}_z$-fat tree via the order-isomorphism from $z$ to $\otp(z)$.

\begin{definition}
	Suppose that $a = \langle (\beta_\ell, y_\ell) \mid \ell < m \rangle$ is a
	stem for $\bb{P}$, with $\beta_0 < \ldots < \beta_{\ell - 1}$. A \emph{$(\vec{U},
	a)$-fat sequence of trees} is a sequence $\langle (\mathcal{T}_\ell, \mathcal{B}_\ell)
	\mid \ell \leq m \rangle$ such that
	\begin{enumerate}
		\item $\mathcal{T}_0 = \{T_\emptyset\}$, where $T_\emptyset$ is a $\vec{U}$-fat
		tree below $(\beta_0, y_0)$ (if $m = 0$, then $T_\emptyset$ is simply
		a $\vec{U}$-fat tree);
		\item $\mathcal{B}_0 = \{\langle b \rangle \mid b \text{ is a maximal
		element of } T_\emptyset\}$;
		\item for each $0 < \ell \leq m$, $\mathcal{T}_\ell = \{T_{\vec{b}} \mid
		\vec{b} \in \mathcal{B}_{\ell - 1}\}$, where
		\begin{enumerate}
			\item if $\ell < m$, then each $T_{\vec{b}}$ is a $\vec{U}$-fat tree above
			$(\beta_{\ell-1}, y_{\ell-1})$ and below $(\beta_{\ell}, y_{\ell})$;
			\item if $\ell = m$, then each $T_{\vec{b}}$ is a $\vec{U}$-fat tree above
			$(\beta_{\ell-1}, y_{\ell-1})$;
		\end{enumerate}
		\item for each $0 < \ell \leq m$, $\mathcal{B}_\ell$ is the set of all
		sequences $\vec{b} = \langle b_i \mid i \leq \ell \rangle$ such that
		\begin{enumerate}
			\item $\vec{b}^- := \langle b_i \mid i < \ell \rangle$
			is an element of $\mathcal{B}_{\ell - 1}$;
			\item $b_\ell$ is a maximal element of $T_{\vec{b}^-}$.
		\end{enumerate}
	\end{enumerate}
\end{definition}

Notice that, if $a = \langle (\beta_\ell, y_\ell) \mid \ell < m \rangle$ is a
stem for $\bb{P}$ and $\langle (\mathcal{T}_\ell, \mathcal{B}_\ell)
\mid \ell \leq m \rangle$ is a $(\vec{U}, a)$-fat sequence of trees, then every
$\vec{b} = \langle b_i \mid i \leq m \rangle$ in $\mathcal{B}_m$ determines
a stem $a_{\vec{b}}$ for $\bb{P}$ defined by letting
\[
	a_{\vec{b}} = b_0 ^\frown \langle (\beta_0, y_0) \rangle ^\frown b_1 ^\frown
	\langle (\beta_1, y_1) \rangle ^\frown \ldots ^\frown \langle (\beta_{m-1},
	y_{m-1}) \rangle ^\frown b_m.
\]
Note also that since, for each $\ell < m$ and each $T \in \mathcal{T}_\ell$,
we have that $T$ is fat below $(\beta_\ell, y_\ell)$, and since $\kappa$ is
strongly inaccessible, it follows that $|\mathcal{T}_m| < \kappa$.

The following fact is easily verified; the element $\vec{b}$ is constructed
by recursion, taking advantage of the fact that branching in fat trees occurs
on measure-one sets.

\begin{fact} \label{possible_fact}
	Suppose that $(a, A) \in \bb{P}$ and $\langle (\mathcal{T}_\ell, \mathcal{B}_\ell)
	\mid \ell \leq m \rangle$ is a $(\vec{U}, a)$-fat sequence of trees. Then there is
	$\vec{b} \in \mathcal{B}_m$ such that the stem $a_{\vec{b}}$ is possible for
	$(a, A)$.
\end{fact}

\begin{lemma} \label{denseSetLemma}
	Suppose that $p = (a, A) \in \bb{P}$ and $D \subseteq \bb{P}$ is a dense
	open set below $p$.
  Suppose that $a = \langle (\beta_\ell, y_\ell) \mid \ell < m \rangle$ is such that, for
	all $\ell_0 < \ell_1 < m$, $\beta_{i_0} < \beta_{i_1}$. Then there is a
	$(\vec{U}, a)$-fat sequence of trees $\langle (\mathcal{T}_\ell, \mathcal{B}_\ell) \mid
	\ell \leq m \rangle$ such that, for all $\vec{b} \in \mathcal{B}_m$, if
	$a_{\vec{b}}$ is possible for $p$, then there is $B$ such that
	$(a_{\vec{b}}, B) \leq p$ and $(a_{\vec{b}}, B) \in D$.
\end{lemma}

\begin{proof}
	Let us first argue that it suffices to prove the lemma in the case
	$m = 0$, i.e., for conditions with an empty stem. To this end, suppose
	that $m > 0$, and let $(\beta, y) = (\beta_{m-1}, y_{m-1})$. By the Factorization Lemma
	(\ref{factorizationLemma}), we have $\bb{P}/p \cong \bb{P}_{\vec{U}_y}/p_0
	\times \bb{P}_{\vec{U}, \beta + 1}/p_1$, where $p_1 = (\emptyset, A \restriction
	[\beta + 1, \theta))$ and $p_0$ is of the form
	$(\bar{a}, \bar{B})$, where $\bar{a} = \langle (\bar{\beta}_\ell, \bar{y}_\ell)
	\mid \ell < m-1 \rangle$ is such that, for all $\ell < m-1$, $\bar{y}_\ell$ is
	the image of $y_\ell$ under the order-preserving isomorphism between
	$y$ and $\otp(y)$. We may assume by induction that we have established
	the full lemma for the forcing $\bb{P}_{\vec{U}_y}$. Let us also assume that
	we have established the lemma for $\bb{P}_{\vec{U}, \beta + 1}$ for conditions
	with empty stems. Let us regard $D$ as a dense subset of
	$\bb{P}_{\vec{U}_y} \times \bb{P}_{\vec{U}, \beta + 1}$ in the natural way.

	Let $\dot{D}_0$ be a $\bb{P}_{\vec{U}, \beta + 1}$-name for the set of
	$q_0 \in \bb{P}_{\vec{U}_y}$ such that, for some $q_1 \in
	\dot{G}_{\bb{P}_{\vec{U}, \beta + 1}}$, $(q_0, q_1) \in D$. Then
	$\dot{D}_0$ is forced to be a dense open subset of $\bb{P}_{\vec{U}_y}$ below
	$p_0$. By repeated applications of Lemma \ref{prikryLemma}, we can find a condition
	$p_1^* = (\emptyset, A^*) \leq^* p_1$ in $\bb{P}_{\vec{U}, \beta + 1}$
	and a dense open subset $D_0$ of $\bb{P}_{\vec{U}_y}/p_0$ such that
	$p_1^* \Vdash ``\dot{D}_0 = \check{D}_0"$.

	Apply the inductive hypothesis to $\bb{P}_{\vec{U}_y}$, $p_0$, and $D_0$
	to find a $(\vec{U}_y, \bar{a})$-fat sequence of trees
	$\langle (\bar{\mathcal{T}}_\ell, \bar{\mathcal{B}}_\ell) \mid \ell < m \rangle$ such that,
	for all $\vec{b} \in\bar{\mathcal{B}}_{m-1}$, if $a_{\vec{b}}$ is possible for
	$p_0$, then there is $\bar{B}_{\vec{b}}$ such that $(a_{\vec{b}}, \bar{B}_{\vec{b}})
	\leq p_0$ and $(a_{\vec{b}}, \bar{B}_{\vec{b}}) \in D_0$.

	Now, for each such $\vec{b} \in \bar{\mathcal{B}}_{m-1}$, the set of
	$p \in \bb{P}_{\vec{U}, \beta + 1}$ such that $((a_{\vec{b}}, B_{\vec{b}}),
	p) \in D$ is dense below $p_1^*$. Denote this set by $D_{1, \vec{b}}$.
	We can now apply the lemma for $\bb{P}_{\vec{U}, \beta + 1}$ for
	conditions with empty stems to the condition $p_1^*$ and the set $D_{1, \vec{b}}$
	to find a $\vec{U}$-fat tree $T_{\vec{b}}$ such that
	\begin{itemize}
		\item for all $\langle (\alpha_i, x_i) \mid i \leq k \rangle \in T_{\vec{b}}$,
		we have $\alpha_0 > \beta$;
		\item for all maximal $c \in T_{\vec{b}}$, if $c$ is possible for $p_1^*$,
		then there is $C_{\vec{b}}$ such that $(c, C_{\vec{b}}) \leq p_1^*$
		and $(c, C_{\vec{b}}) \in D_{1, \vec{b}}$.
	\end{itemize}
	By thinning out $T_{\vec{b}}$ if necessary, we may assume that it is
	$\vec{U}$-fat above $(\beta, y)$.

	Let $\langle (\mathcal{T}_\ell, \mathcal{B}_\ell) \mid \ell < m \rangle$
	be the sequence obtained in the natural way by applying the order-isomorphism
	between $\otp(y)$ and $y$ to the $(\vec{U}_y, \bar{a})$-fat sequence
	$\langle (\bar{\mathcal{T}}_\ell, \bar{\mathcal{B}}_\ell) \mid \ell < m \rangle$.
	Let $\mathcal{T}_m = \{T_{\vec{b}} \mid \vec{b} \in \bar{\mathcal{B}}_{m-1}\}$,
	and let $\mathcal{B}_m$ be the set of sequences of the form
	$\hat{b} ^\frown \langle c \rangle$, where $\hat{b} \in \mathcal{B}_{m-1}$ and,
	letting $\vec{b}$ be the isomorphic copy of $\hat{b}$ in $\bar{\mathcal{B}}_{m-1}$,
	we have that $c$ is a maximal element of $T_{\vec{b}}$. Then it is easily
	verified that $\langle (\mathcal{T}_\ell, \mathcal{B}_\ell) \mid \ell \leq m
	\rangle$ is a $(\vec{U}, a)$-fat sequence of trees as in the statement of the
	lemma.

	It therefore suffices to consider $p$ of the form $(\emptyset, A)$.
  We thus need to find a single $\vec{U}$-fat tree $T$ such that, for every maximal element
	$b$ of $T$, if $b$ is possible for $p$, then there is $B$ such that
	$(b, B) \leq p$ and $(b, B) \in D$.

	We accomplish this by inductively constructing a decreasing sequence
  of conditions $\langle (\emptyset, A_n) \mid n < \omega \rangle$. Intuitively, $A_n$ will
  take care of extensions $(b, B) \leq (\emptyset, A)$ such that $|b| = n$. We explicitly go
  through the first few steps of the construction.

	Let $A_0 = A$. If there is a direct extension of $(\emptyset, A)$ in $D$, then we are done
  by setting $T = \{\emptyset\}$. Thus, suppose there is no such direct extension. For every
  stem $a$ possible for $(\emptyset, A_0)$ and every $\alpha \in (\gamma^a, \theta)$, let
  $Y_{0, a, \alpha} = \{x \in A_0(\alpha) \mid a(\gamma^a) \prec x$ and $\gamma^a \in Z^\alpha_x \}$.
  Let $Y^0_{0, a, \alpha} = \{x \in Y_{0, a, \alpha} \mid$ for some $B$, $(a^\frown (\alpha, x), B)
  \in D \}$, and let $Y^1_{0, a, \alpha} = Y_{0, a, \alpha} \setminus Y^0_{0, a, \alpha}$. Find
  $i(0, a, \alpha) < 2$ such that $Y^{i(0, a, \alpha)}_{0, a, \alpha} \in U_\alpha$, and let
  $Y^*_{0, a, \alpha} = Y^{i(0, a, \alpha)}_{0, a, \alpha}$. For $\alpha < \theta$, let $A_1(\alpha)$
  be the set of $x \in A_0(\alpha)$ such that, for all stems $a$ possible for $(\emptyset, A_0)$
  such that $a(\gamma^a) \prec x$ and $\gamma^a \in Z^\alpha_x$, $x \in Y^*_{0, a, \alpha}$.
  By the Diagonal Intersection Lemma (\ref{diagonal_intersection_lemma}), $A_1(\alpha) \in U_\alpha$ for all $\alpha < \theta$,
  so $(\emptyset, A_1) \leq^* (\emptyset, A_0)$. Note that $(\emptyset, A_1)$ satisfies the following
  property, which we denote $(*)_1$:

	\begin{quote}
		Suppose that $q = (a ^\frown (\alpha, x), B) \leq (\emptyset, A_1)$ and $q \in D$. Then, for every
    $y \in A_1(\alpha)$ such that $a(\gamma^a) \prec y$ and $\gamma^a \in Z^\alpha_y$, there is $B_y$
    such that $(a ^\frown (\alpha, y), B_y) \in D$.
	\end{quote}

	Now suppose that there is a stem $a = \{(\alpha, x)\}$ possible for $(\emptyset, A_1)$ and a $B$ such that
  $(a, B) \in D$. We can then define a $\vec{U}$-fat tree $T$ of height $1$
	whose maximal elements are all
  $\langle (\alpha, x) \rangle$ such that $x \in A_1(\alpha)$. We are then done, as $T$ easily
  satisfies the requirements of the lemma. Thus, suppose there is no such $a$ and proceed to define
  $(\emptyset, A_2)$ as follows.

	For every stem $a$ possible for $(\emptyset, A_1)$ and every $\alpha \in (\gamma^a, \theta)$, let
  $Y_{1, a, \alpha} = \{x \in A_1(\alpha) \mid a(\gamma^a) \prec x$ and $\gamma^a \in Z^\alpha_x \}$.
  Let $Y^0_{1, a, \alpha}$ be the set of all $x \in Y_{1, a, \alpha}$ such that there are
  $\beta^\alpha_x \in (\alpha, \theta)$ and $W^\alpha_x \in U_{\beta^\alpha_x}$ such that, for all
  $y \in W^\alpha_x$:
	\begin{itemize}
		\item{$x \prec y$ and $\alpha \in Z^{\beta^\alpha_x}_y$;}
		\item{there is $B$ such that $(a^\frown (\alpha, x) ^\frown (\beta^\alpha_x, y), B) \in D$.}
	\end{itemize}
	Let $Y^1_{1, a, \alpha} = Y_{1, a, \alpha} \setminus Y^0_{1, a, \alpha}$. Find $i(1, a, \alpha) < 2$
  such that $Y^{i(1, a, \alpha)}_{1, a, \alpha} \in U_\alpha$, and let $Y^*_{1, a, \alpha} =
  Y^{i(1, a, \alpha)}_{1, a, \alpha}$. For $\alpha < \theta$, let $A_2(\alpha)$ be the set of
  $x \in A_1(\alpha)$ such that, for all stems $a$ possible for $(\emptyset, A_1)$ such that
  $a(\gamma^a) \prec x$ and $\gamma^a \in Z^\alpha_x$, $x \in Y^*_{1, a, \alpha}$.
	By the Diagonal Intersection Lemma (\ref{diagonal_intersection_lemma}), we
	have $A_2(\alpha) \in U_\alpha$ for all $\alpha < \theta$. Then
  $(\emptyset, A_2) \leq^* (\emptyset, A_1)$, and $(\emptyset, A_2)$ satisfies the
	following property, which we denote $(*)_2$:

	\begin{quote}
		Suppose that $q = (a ^\frown (\alpha, x) ^\frown (\beta, y), B) \leq (\emptyset, A_2)$ and $q \in D$.
    Then, for every $x' \in A_2(\alpha)$ such that $a(\gamma^a) \prec x'$ and $\gamma^a \in Z^\alpha_{x'}$,
    there is $\beta^\alpha_{x'} \in (\alpha, \theta)$ and $W^\alpha_{x'} \in U_{\beta^\alpha_{x'}}$
    such that, for all $y' \in W^\alpha_{x'}$, there is $B'$ such that $(a ^\frown (\alpha, x') ^\frown
    (\beta^\alpha_{x'}, y'), B') \in D$.
	\end{quote}

	Suppose that there is a stem $a = \{(\alpha, x), (\beta, y)\}$ possible for $(\emptyset, A_2)$ with $\alpha < \beta$ and a
  $B$ such that $(a, B) \in D$. Using $(*)_2$, we can define a $\vec{U}$-fat tree
	$T$ of height $2$ whose maximal
  elements are all $\langle (\alpha, x'), (\beta^\alpha_{x'}, y') \rangle$ such that $x' \in
  A_2(\alpha)$ and $y' \in W^\alpha_{x'}$. We are then done, as $T$ satisfies the requirements of the
  lemma. If there is no such stem $a$, then continue in the same manner.

	In this way, we can construct $A_n$ such that, if there is a stem $a$ possible for $(\emptyset, A_n)$
  with $|a| = n$ and a $B$ such that $(a, B) \leq (\emptyset, A_n)$ and
	$(a, B) \in D$, then there is a $\vec{U}$-fat tree of height $n$ as desired.
  For $\alpha < \theta$, let $A_\infty(\alpha) = \bigcap_{n < \omega} A_n(\alpha)$. For all $n < \omega$,
  $(\emptyset, A_\infty) \leq^* (\emptyset, A_n)$. Find $(a, B) \leq (\emptyset, A_\infty)$ such that
  $(a, B) \in D$. Let $n^* = |a|$. Then $a$ is possible for $(\emptyset, A_{n^*})$, so there is a fat
  tree of height $n^*$ as required by the lemma.
\end{proof}

Suppose that $T$ is a fat tree, $\gamma < \theta$, and $z \in \mathcal{P}_\kappa(\kappa^{+\gamma})$.
$T \restriction (\gamma, z)$ is the subtree of $T$ consisting of all $\langle (\alpha_i, x_i)
\mid i \leq k \rangle \in T$ such that, for all $i \leq k$, $\alpha_i \in Z^\gamma_z$ and $x_i
\prec z$. Let $\gamma_T := \sup(\{\alpha \mid$ for some $\langle (\alpha_i, x_i) \mid i \leq
k \rangle \in T$ and $i \leq k$, $\alpha = \alpha_i\})$. Note that,
if $\theta$ is weakly inaccessible and $|\mathcal{P}_\kappa(\kappa^{+\alpha})| < \theta$
for all $\alpha < \theta$, then $\gamma_T < \theta$.

\begin{theorem} \label{regular_thm}
	If $\theta$ is weakly inaccessible and $|\mathcal{P}_\kappa(\kappa^{+\alpha})| < \theta$ for all
  $\alpha < \theta$, then $\kappa$ remains regular in $V^{\bb{P}}$.
\end{theorem}

\begin{proof}
	Let $p = (a, A) \in \bb{P}$, let $\delta < \kappa$,
  and suppose that $\dot{f}$ is a $\bb{P}$-name forced by $p$ to be a function from $\delta$ to $\kappa$.
  We will find $q \leq p$ forcing the range of $\dot{f}$ to be bounded below $\kappa$.

	For all $\xi < \delta$, let $D_\xi$ be the set of $(b, B) \in \bb{P}$ such that $(b, B) \Vdash
  ``\dot{f}(\xi) < \kappa_{b(\gamma^b)}."$ Each $D_\xi$ is a dense, open subset of $\bb{P}$
	below $p$. For
  $\xi < \delta$, let $S_\xi$ be the set of stems $b$ such that, for some $B$, $(b, B) \leq p$ and
  $(b, B) \in D_\xi$. For all $b \in S_\xi$, fix a $B^\xi_b$ witnessing this. For each $\beta \in
  (\gamma^a, \theta)$, let $A^*(\beta)$ be the set of $y \in A(\beta)$ such that, for all $\xi < \delta$
  and all $b \in S_\xi$ such that $b(\gamma^b) \prec y$ and $\gamma^b \in Z^\beta_y$, $y \in B^\xi_b(\beta)$. By the Diagonal Intersection Lemma (\ref{diagonal_intersection_lemma}),
	$A^*(\beta) \in U_\beta$.
  For $\beta \in \dom(A) \cap \gamma^a$, let $A^*(\beta) = A(\beta)$. Then $(a, A^*) \leq (a, A)$.
  Let $R$ be the set of stems possible for $(a, A^*)$. For $\gamma < \theta$, let $R_{<\gamma} =
  \{c \in R \mid \gamma^c < \gamma\}$. For all $c \in R$, let
	$p_c = (a, A^*) \downarrow c$. For all
  $c \in R$ and $\xi < \delta$, apply Lemma \ref{denseSetLemma} to $p_c$ and $D_\xi$ to
	obtain a $(\vec{U}, c)$-fat sequence of trees $\langle (\mathcal{T}_{c, \xi, \ell},
	\mathcal{B}_{c, \xi, \ell}) \mid \ell \leq |c| \rangle$. Let $E$ be the set of
  limit ordinals $\gamma < \theta$ such that, for all $c \in R_{<\gamma}$, all $\xi < \delta$,
	and all $T \in \mathcal{T}_{c, \xi, |c|}$, we have $\gamma_T < \gamma$.
  Then $E$ is club in $\theta$. Fix $\gamma \in E \setminus
  (\gamma^a + 1)$.

	\begin{claim}
		Let $Y_\gamma$ be the set of $z \in A^*(\gamma)$ such that, for all
		$c \in R_{{<}\gamma}$ such that $c(\gamma^c) \prec z$ and $\gamma^c \in
		Z^\gamma_z$, for all $\xi < \delta$, and for all $T \in \mathcal{T}_{c,
		\xi, |c|}$, we have that $T \restriction (\gamma, z)$ is $\vec{U}$-fat below
		$(\gamma, z)$ and has the same height as $T$. Then $Y_\gamma \in U_\gamma$.
	\end{claim}

	\begin{proof}
		Let $j = j_\gamma$. We show that $j``\kappa^{+\gamma} \in j(Y_\gamma)$.
		Note that, for all $c \in j(R_{{<}\gamma})$, if $c(\gamma^c) \prec
		j``\kappa^{+\gamma}$ and $\gamma^c \in j``\gamma$, then there is a stem
		$\bar{c} \in R_{{<}\gamma}$ such that $c = j(\bar{c})$. Fix such a $c$
		and an ordinal $\xi < \delta$. Since $|\mathcal{T}_{\bar{c}, \xi, |\bar{c}|}|
		< \kappa$, we have $j(\mathcal{T}_{\bar{c}, \xi, |\bar{c}|}) =
		\{j(T) \mid T \in \mathcal{T}_{\bar{c}, \xi, |\bar{c}|}\}$. If
		$T \in \mathcal{T}_{\bar{c}, \xi, |\bar{c}|}$, then $j(T) \restriction
		(j(\gamma), j``\kappa^{+\gamma}) = j``T$, which is fat below $(j(\gamma),
		j``\kappa^{+\gamma})$ of the same height as $j(T)$. Hence, $j``\kappa^{+\gamma} \in j(Y_\gamma)$.
	\end{proof}

	Choose $z \in Y_\gamma$. Then $q = (a^\frown (\gamma, z), A^{**}) \leq (a, A^*)$, where
  $A^{**}(\alpha) = A^*(\alpha)$ for all $\alpha \in \dom(A^*) \cap \gamma^a$,
	$A^{**}(\alpha) = \{x \in A^*(\alpha) \mid x \prec z\}$ for all
  $\alpha \in (\gamma^a, \gamma)$, and
	$A^{**}(\alpha) = \{x \in A^*(\alpha) \mid z \prec x\}$ for all $\alpha
  \in (\gamma, \theta)$.

	We claim that $q$ forces the range of $\dot{f}$ to be bounded below $\kappa_z$. Suppose
  for sake of contradiction that there is $\xi < \delta$ and $r \leq q$ such that $r \Vdash
  ``\dot{f}(\xi) \geq \kappa_z."$ Let $r = (d, F)$, and let $c = \{(\alpha, x) \in d \mid \alpha
  < \gamma\}$. Then $c \in R_{<\gamma}$, $c(\gamma^c) \prec z$, and $\gamma^c \in Z^\gamma_z$,
  so $T \restriction (\gamma, z)$ is fat below $(\gamma, z)$, of the same height
	as $T$, for every $T \in \mathcal{T}_{c, \xi, |c|}$. Applying Fact
	\ref{possible_fact} inside $\bb{P}_{\vec{U}_z}$, it follows that we can
	find $\vec{b} \in \mathcal{B}_{c, \xi, |c|}$ such that $a_{\vec{b}} \cup d$
	is possible for $r$. By the definition of $\langle (\mathcal{T}_{c, \xi, \ell},
	\mathcal{B}_{c, \xi, \ell}) \mid \ell \leq |c| \rangle$, it follows that
	there is $B'$ such that $(a_{\vec{b}}, B') \in D_\xi$. Moreover, by our
	construction of $A^*$, we may assume that, for all $\alpha \in \dom(B')
	\cap (\gamma^{a_{\vec{b}}})$, we have $B'(\alpha) = A^*(\alpha)$. All of
	this together means that $(a_{\vec{b}}, B')$ and $r$ are compatible.
	However, as $(a_{\vec{b}}, B') \in D_\xi$, we have $(a_{\vec{b}}, B') \Vdash
  ``\dot{f}(\xi) < \kappa_{a_{\vec{b}}(\gamma^{a_{\vec{b}}})} < \kappa_z,"$ contradicting
	the assumption that $r \Vdash ``\dot{f}(\xi) \geq \kappa_z."$
\end{proof}

\section{Cardinal arithmetic} \label{arithmetic_section}

For the rest of the paper, we will let $\theta$ be the least weakly inaccessible
cardinal above $\kappa$. In this section, we show that, with this assumption,
if $G$ is $\bb{P}$-generic over $V$, then, in
$V[G]$, every limit point of $C_G$ below $\kappa$ is a singular strong limit cardinal at
which SCH fails. We begin by making the following definition.

\begin{definition}
	For $\beta<\theta$ and $y \in X_\beta$, we write
	$o(y)$ for $f_\beta^\beta(\kappa_y) (= \otp(Z^\beta_y))$.
\end{definition}

We note that $o(y)$ formally depends on $\beta$ as well as $y$, but, as
the value of $\beta$ will always be clear from context, we suppress its mention.
For $\nu<\kappa$, we let $\theta(\nu)$ be the least weakly inaccessible cardinal
greater than $\nu$.  Using this notation we have $o(y) < \theta(\kappa_y)$.
Also, by the preparation of our ground model, we have $2^\kappa \geq \theta$
and, for all limit ordinals $\beta < \kappa$, we have $|\bigcup_{\alpha < \beta}
\mathcal{P}_\kappa (\kappa^{+\alpha})| = \kappa^{+\beta}$. As a result,
for all limit ordinals $\beta < \theta$, the following statements hold for almost
all $y \in X_\beta$:
\begin{itemize}
	\item $2^{\kappa_y} \geq \theta(\kappa_y)$;
	\item $|\bigcup_{\alpha < o(y)}\mathcal{P}_{\kappa_y}(\kappa_y^{+\alpha})|
	= \kappa_y^{+o(y)}$.
\end{itemize}
We henceforth assume that in fact all $y \in X_\beta$ satisfy these two conditions.

\begin{lemma} \label{knaster_lemma}
	Suppose that $\beta < \theta$ is a limit ordinal and $y \in X_\beta$.
	Then $\P_{\vec{U}_y}$, as defined before Lemma
	\ref{factorizationLemma}, has the $\kappa_y^{+o(y)+1}$-Knaster
property.  \end{lemma}

\begin{proof}
  It is not hard to see that there are just $\kappa_y^{+o(y)}$ many
  stems in this poset and that conditions with the same
  stem are compatible.
\end{proof}

Suppose now that $G$ is $\bb{P}$-generic over $V$, and fix $\nu \in \lim(C_G)$. Also
fix $p \in G$, $\beta \in \dom(a^p)$, and $y \in X_\beta$ such that
$a^p(\beta) = y$ and $\kappa_y = \nu$. Since $\nu \in \lim(C_G)$, we know
that $\beta$ is a limit ordinal.

\begin{lemma}
	$\nu$ is a strong limit cardinal in $V[G]$.
\end{lemma}

\begin{proof}
	Fix $\mu < \nu$, and let $\beta_0 < \beta$ be the largest limit ordinal such
	that there exists $q \in G$ with $\beta_0 \in \dom(a^q)$ and
	$\kappa_{a^q(\beta_0)} \leq \mu$, if such a $q$ exists. Let $\beta_0 = 0$
	otherwise. Let $\beta_1 < \beta$ be the least ordinal such that there exists
	$q \in G$ such that $\beta_1 \in \dom(a^q)$ and $\mu < \kappa_{a^q(\beta_1)}$.
	Note that $\beta_1$ is a successor ordinal and there are only finitely
	many ordinals $\gamma$ between $\beta_0$ and $\beta_1$ for which there
	exists $q \in G$ such that $\gamma \in \dom(a^q)$.
	We can now find $q \in G$ such that
	\begin{itemize}
		\item $\beta_1 \in \dom(a^q)$ and, if $\beta_0 > 0$, then $\beta_0
		\in \dom(a^q)$ as well;
		\item for all $\gamma$ in the interval $[\beta_0, \beta_1]$,
		either $\gamma \in a^q$ or $A^q(\gamma) = \emptyset$.
	\end{itemize}

	If $\beta_0 = 0$, then, in $V$ the direct ordering $\leq^*$ is $\mu^+$-closed
	in $\bb{P}/q$, so, by Lemma \ref{prikryLemma}, forcing with $\bb{P}$ below
	$q$ does not add any new subsets to $\mu$. Since $\nu$ was strongly inaccessible
	and hence strong limit in $V$, it follows that $2^\mu < \nu$ continues to hold
	in $V[G]$.

	If $\beta_0 > 0$, then let $y_0 = a^q(\beta_0)$. By the Factorization Lemma
	(\ref{factorizationLemma}), $\bb{P}/q \cong \bb{P}_{\vec{U}_{y_0}}/q_0
	\times \bb{P}_{\vec{U}, \beta_0 + 1}/q_1$ for some conditions $q_0$ and $q_1$.
	As in the case in which $\beta_0 = 0$, forcing with $\bb{P}_{\vec{U}, \beta_0 + 1}/q_1$
	does not add any new subsets to $\mu$. Moreover, in $V$,
	we have $|\bb{P}_{\vec{U}_{y_0}}| < \nu$ and $\nu$ is strongly inaccessible,
	so forcing with $\bb{P}_{\vec{U}_{y_0}}$ cannot add $\nu$-many distinct
	subsets to $\mu$. Again, it follows that $2^\mu < \nu$ continues to hold in $V[G]$.
\end{proof}

The argument of the above proof easily adapts to yield, together with Theorem
\ref{regular_thm}, the following corollary.

\begin{corollary}
	$\kappa$ is strongly inaccessible in $V[G]$.
\end{corollary}

We next argue that $\nu$ is singular in $V[G]$. The proof breaks into two cases,
depending on whether or not $\cf^V(\beta) < \kappa$.

Suppose first that $\cf^V(\beta) < \kappa$. Then, by Lemma \ref{nicenessLemma},
we have that $\cf^V(\beta) < \kappa_y = \nu$ and $Z^\beta_y$ is unbounded in
$\beta$. It follows by genericity that the set $A := \{\alpha \in Z^\beta_y \mid
\exists p \in G [\alpha \in a^p]\}$ is unbounded in $\beta$ and that
\[
	\nu = \sup\{\kappa_x \mid \exists \alpha \in A \exists p \in G [a^p(\alpha) = x]\}.
\]
Therefore, in $V[G]$, we have $\cf(\nu) = \cf(\beta) < \nu$.

Next, suppose that $\cf^V(\beta) \geq \kappa$. The following lemma shows
that $\cf^{V[G]}(\nu) = \omega$.

\begin{lemma} \label{cofomega} If $\cf^V(\beta) \geq \kappa$,
then $\nu$ and $\nu^{+o(y)}$ change their cofinality to $\omega$ in $V[G]$.
\end{lemma}

\begin{proof}
  Work in $V$. Using the Factorization Lemma (\ref{factorizationLemma}), $\P_{\vec{U}}/p \cong \P_{\vec{U}_y}/p_0
  \times \P_{\vec{U}, \beta+1}/p_1$ for some $p_0$ and $p_1$. As $\po_{\vec{U},\beta+1}/p_1$ does
  not add new bounded subsets to $\theta(\nu)$, it is sufficient to focus on the forcing
  $\po_{\vec{U}_y}$ which adds a Radin club to $\nu$. For notational simplicity, let
  $\delta = f^\beta_\beta(y) = o(y)$, and let $\vec{U}_y = \langle V_\xi \mid \xi < \delta \rangle$.
  By Lemma \ref{nicenessLemma}, we have $\rho := \cf(\delta) \geq \nu$ and $\delta <
	\theta(\nu)$. We will show that $\nu$ and $\nu^{+\delta}$
  change their cofinalities to $\omega$ after forcing with $\P_{\vec{U}_y}$.

  Choose an increasing continuous sequence $\vec{\delta} = \la \delta_\alpha \mid
  \alpha < \rho\ra$ cofinal in $\delta$. Let $G_y$ be $\P_{\vec{U}_y}$-generic
  over $V$. For every $\beta' < \delta$, let $\alpha(\beta')
  < \rho$ be the minimal $\alpha < \rho$ so that $\beta'
  \leq \delta_\alpha$.

  Since $\delta < \theta(\nu)$, we have $\nu^{+\delta} > \rho$.
  Let $\alpha_0 < \rho$ be the least ordinal so that
  $\rho < \nu^{+\delta_{\alpha_0}}$.  By reindexing, we can assume that $\alpha_0 = 0$.
  Note that, for every $\beta'$ with $\delta_0 \leq \beta' <
  \delta$, $Y_{\beta'} := \{x \in \mathcal{P}_\nu(\nu^{+{\beta'}})
	\mid \alpha(\beta') \in x\}$ belongs to
  $V_{\beta'}$. For $x \in \mathcal{P}_\nu(\nu^{+{\beta'}})$, let
	$\nu_x := x \cap \nu$.

  Move now to $V[G_y]$. Given $x \in C^{\spc}_{G_y}$, let $\beta(x)$ be the
	unique $\beta' < \delta$ such that there exists $q \in G_y$ for which
	$a^q(\beta') = x$, and let $\alpha(x) = \alpha(\beta(x))$.
	By the previous paragraph, there is some
	$\nu_0 \in C_{G_y}$ such that, for every $x \in
  C^{\spc}_{G_y}$, if $\nu_x > \nu_0$, then $x \in Y_{\beta(x)}$,
	and hence $\alpha(x) \in x$.
	Let $x_0$ be the minimal $x \in C^{\spc}_{G_y}$
  satisfying the above.  Starting from $x_0$, we define a sequence $\vec{x} = \la
  x_n \mid n < \omega\ra \subset C^{\spc}_{G_y}$.  For each $n < \omega$, let
  $x_{n+1}$ be the minimal $x$ above $x_n$ in $C^{\spc}_{G_y}$ so that
	$\sup(x_{n} \cap \rho) < \alpha(x) < \rho$. Let $\nu_\omega =
	\bigcup_{n<\omega} \nu_{x_n}$.

	We claim that $\nu_\omega = \nu$. Suppose otherwise.  Then $\nu_\omega =
  \nu_x$ for some $x \in C^{\spc}_{G_y}$. Let $\alpha = \alpha(x)$, and note
	that $\alpha \in x$. Since
  $\la \nu_{x_n} \mid n < \omega\ra$ is cofinal in $\nu_\omega$, we have $x
  \cap \rho = \bigcup_{n<\omega} (x_n \cap \rho)$.  There is thus some $m < \omega$ such that
	$\alpha \in x_m$. But $\alpha \geq \alpha(x_{m+1}) > \sup(x_m \cap \rho)$,
	which is a contradiction.

  It follows that $\nu = \nu_\omega$, so $\nu$ changes its cofinality to
  $\omega$.  The set $C^{\spc}_{G_y} \subset \power_\nu(\nu^{+\delta})$ is
  $\subseteq$-cofinal in $\power_\nu(\nu^{+\delta})$. Since $\la \nu_{x_n}
  \mid n < \omega\ra$ is cofinal in $\nu$, $\la x_n \mid n <
  \omega\ra$ is $\subset$-cofinal in $C^{\spc}_{G_y}$. Thus, $\nu^{+\delta} = \bigcup_{n < \omega}
  x_n$.  It follows that $\cf(\nu^{\delta}) = \omega$, as each $x_n$ is bounded
  in $\nu^{+\delta}$.
\end{proof}

\begin{remark}
	It follows easily from the above proof that, if $\cf^V(\beta) \geq \kappa$, then
	all $V$-regular cardinals between $\nu$ and $\nu^{+o(y)}$ change their
	cofinality to $\omega$ in $V[G]$, as well.
\end{remark}

In either case, we have shown that $\nu$ is a singular strong limit cardinal in
$V[G]$. To show that SCH fails at $\nu$, we first make the following observation.

\begin{lemma} \label{new_successor_lemma}
	$(\nu^+)^{V[G]} = (\nu^{+o(y) + 1})^V$.
\end{lemma}

\begin{proof}
	Exactly as in the start of the proof of Lemma \ref{cofomega}, since
	$\bb{P}_{\vec{U}, \beta + 1}/p_1$ does not add new subsets to
	$\theta(\nu)$, it will suffice to show that forcing with
	$\bb{P}_{\vec{U}_y}$ collapses all $V$-cardinals in the interval
	$(\nu, (\nu^{+o(y)})^V]$ and preserves all cardinals greater than or equal to
	$\nu^{+o(y) + 1}$. Observe first that, in $V[G_y]$, we have
	$\nu^{+o(y)} = \bigcup C^{\spc}_{G_y}$. Since $C^{\spc}_{G_y}$ is
	a $\prec$-increasing sequence of cofinality $\cf^{V[G]}(\nu) < \nu$,
	consisting sets of cardinality less than $\nu$, it follows that
	$|(\nu^{+o(y)})^V| = \nu$ in $V[G]$. Next note that, in $V$, Lemma
	\ref{knaster_lemma} implies that $\bb{P}_{\vec{U}_y}$ has
	the $\nu^{+o(y) + 1}$-Knaster property and hence preserves all cardinals
	greater than or equal to $\nu^{+o(y) + 1}$.
\end{proof}

To conclude that SCH fails at $\nu$ in $V[G]$, it is now enough to observe that,
in $V$, we have $2^\nu \geq \theta(\nu)$ and $\theta(\nu)$ is a weakly
inaccessible cardinal greater than $\kappa^{+o(y) + 1}$. It follows that, in
$V[G]$, we still have $2^\nu \geq \theta(\nu) > \kappa^{+o(y) + 1}$, and
$\theta(\nu)$ remains a cardinal. This completes the argument that, in
$V[G]$, every limit point of $C_G$ below $\kappa$ is a singular strong limit
cardinal at which SCH fails.

\section{Approachability} \label{weaksquare}

In this section we characterize precisely which successors $\nu^+$ for
$\nu \in \lim(C_G)$ have reflection properties and then construct the final model
in which the conclusion of Theorem \ref{thm1} will hold. We
begin with the following lemma. We will later need to apply
the lemma to posets $\bb{P}_{\vec{U}_y}$ for $y \in \mathcal{P}_\kappa(\kappa^{+\beta})$,
so note that its proof does not rely on the weak inaccessibility of $\theta$.

\begin{lemma} \label{club_subset}
	Suppose that, in $V$,
	\begin{itemize}
		\item $\delta$ is an ordinal of cofinality $\mu < \kappa$;
		\item $p \in \bb{P}$
		\item $\nu_0 < \mu < \nu_1$ are such that one of the following four alternatives
		holds:
		\begin{itemize}
			\item $p$ forces $\nu_0$ and $\nu_1$ to be successive limit points of
			$\dot{C}_G$, there is $\alpha \in \dom(a^p)$ such that
			$a^p(\alpha) = y_0$, $\kappa_{y_0} = \nu_0$, and $\nu_0^{+o(y)} < \mu < \nu_1$;
			\textbf{or}
			\item $\nu_1 = \kappa$ and $p$ forces that $\nu_0$ is the largest limit
			point of $\dot{C}_G$ (so, in particular, $p$ forces $\dot{C}_G$ to have a
			final segment of order type $\omega$) and there are $\alpha$ and
			$y_0$ as in the previous alternative; \textbf{or}
			\item $\nu_0 = 0$ and $p$ forces $\nu_1$ to be the least limit point of
			$\dot{C}_G$;
			\item $\nu_0 = 0$, $\nu_1 = \kappa$, and $p$ forces that
			$\otp(\dot{C}_G) = \omega$.
		\end{itemize}
		\item $\dot{C}$ is a $\bb{P}$-name forced by $p$ to be a club in $\delta$.
	\end{itemize}
	Then there is a direct extension $p' \leq^* p$ and a club $D$ in $\delta$
	such that $p' \Vdash ``D \subseteq \dot{C}"$.
\end{lemma}

\begin{proof}
	First we show that it is enough to consider $\delta = \mu$.
	Assume for the moment that $\mu < \delta$.  Let $\pi:\mu \to \delta$ be an
	increasing, continuous, and cofinal function. By passing to a name for a
	subset of $\dot{C}$ we can assume that it is forced that $\dot{C}$ is a subset
	of the range of $\pi$. Now a condition will force that there is a ground model
	club contained in $\dot{C}$ if and only if there is a ground model club
	contained in $\pi^{-1}(\dot{C})$.

	So we may assume that $\delta = \mu$. If $p$ forces either that $\nu_1$ is the least limit
	point of $\dot{C}_G$ or that $\otp(\dot{C}_G) = \omega$, then, by
	applying Lemma \ref{prikryLemma} to
	$\bb{P}/p$, we see that forcing with $\bb{P}$ below $p$ does not add any
	bounded subsets to $\nu_1$, so there is in fact a direct extension $p'$ of $p$
	deciding the value of $\dot{C}$.

	Thus, assume we are in one of the first two alternatives, so $p$ forces that
	$\nu_0$ is a limit point of $C_G$ and there is $\alpha \in \dom(a^p)$ such that
	$a^p(\alpha) = y_0$, $\kappa_{y_0} = \nu_0$, and $\nu_0^{+o(y)} < \mu < \nu_1$. Then, by the Factorization Lemma (\ref{factorizationLemma}),
	$\bb{P}/p \cong \bb{P}_{\vec{U}_{y_0}}/p_0 \times \bb{P}_{\vec{U}, \alpha + 1}/p_1$
	for some $p_0$ and $p_1$. Again, forcing with $\bb{P}_{\vec{U}, \alpha + 1}$
	below $p_1$ does not add any bounded subsets to $\nu_1$, so there is a
	direct extension $p_1'$ forcing $\dot{C}$ to be equal to the interpretation of
	some $\bb{P}_{\vec{U}_{y_0}}$-name, $\dot{C}_0$. But then, by the
	$\nu_0^{+o(y) + 1}$-cc of $\bb{P}_{\vec{U}_{y_0}}$, there is a club
	$D$ in $\mu$ such that $p_0 \Vdash ``D \subseteq \dot{C}"$.
\end{proof}

We use Lemma \ref{club_subset} to show that the approachability property fails at
certain points along our Radin club. Recall that we are assuming that
$\theta$ is the least weakly inaccessible cardinal above $\kappa$.

\begin{lemma} \label{app_lemma} Suppose that $\beta$ is a limit ordinal with $\cf^V(\beta) < \kappa$
and $p$ is a condition such that $a^p(\beta) = y$. Then $p$ forces that
$\kappa_y^{+o(y)+1} \notin I[\kappa_y^{+o(y)+1}]$.  \end{lemma}

\begin{proof}
	Assume for a contradiction that (some extension of) $p$ forces
  $\kappa_y^{+o(y)+1} \in I[\kappa_y^{+o(y)+1}]$.  Let $\langle \dot{z}_\gamma \mid \gamma <
	\kappa_y^{+o(y)+1} \rangle$ be a name for the approachability
	witness. Recall that $\kappa_y$ is forced to be singular in the extension
	by $\bb{P}$ and $\kappa_y^{+o(y) + 1}$ is forced to be its successor. Therefore,
	for all limit $\gamma < \kappa_y^{+o(y) + 1}$, $p \Vdash ``\cf(\gamma) < \kappa_y"$.
	As a result, we can assume that the order type of each $\dot{z}_\gamma$ is forced
to be less than $\kappa_y$. By the Factorization Lemma (\ref{factorizationLemma}), $\P_{\vec{U}} / p
\cong \P_{\vec{U}_y}/p_0 \times \P_{\vec{U}, \beta + 1}/p_1$ for some $p_0$
and $p_1$. By the Prikry property, $\P_{\vec{U}, \beta + 1}/p_1$ does not add
any new subsets to $\kappa_y^{+o(y)+1}$, so we may in fact assume that
$\langle \dot{z}_\gamma \mid \gamma < \kappa_y^{+o(y)+1} \rangle$ is a
$\P_{\vec{U}_y}$ name forced by $p_0$ to be a witness to approachability.
Set $\nu: = \kappa_y$ and, for $x \in \mathcal{P}_\nu(\nu^{+o(y)})$, let
$\nu_x := x \cap \nu$ and $\bar{o}(x) = \otp(\{\eta < o(y) \mid \nu^{+\eta} \in x\})$.

Since $y \in X_\beta$, it follows that $\nu$ is $\nu^{+o(y)}$-supercompact.
Since $\cf(\beta) < \kappa$ and hence $\cf(o(y)) < \nu$, it follows that
$\nu$ is in fact $\nu^{+o(y) + 1}$-supercompact.
Let $j:V \to M$ witness that $\nu$ is $\nu^{+o(y) + 1}$-supercompact.
Set $\delta := \sup(j``\nu^{+o(y) + 1})$ and $\mu = \cf(\delta) = \nu^{+o(y) + 1}$.
In $M$, $j(p_0)$ forces that $\delta$ is approachable
with respect to $j(\langle \dot{z}_\gamma \mid \gamma < \nu^{+o(y)+1}
\rangle )$, so there is a $j(\bb{P}_{\vec{U}_y})$-name $\dot{C}$ for a club subset of $\delta$ such that,
for all $\gamma < \delta$, $j(p_0)$ forces that $\dot{C} \cap \gamma$ is enumerated as
$j(\dot{z})_{\gamma'}$ for some $\gamma' < \delta$.

Let $p_0 = (a_0, A_0)$, and let $\eta = \max(\dom(a_0))$ if $a_0 \neq \emptyset$,
or let $\eta = -1$ otherwise. Consider the condition $j(p_0) = (j(a_0), j(A_0))$.
Note that, for all $(\alpha, x) \in A_0$, $j$ fixes $\nu_x$, and
since $\otp(x) < \nu$, we have $j(x) = j``x$ and $j(\bar{o}(x)) = \bar{o}(x)$.
For $\alpha \in (j(\eta), j(o(y)))$, we have that, in $M$, $j(A_0)(\alpha)$ is a
measure-one set for a measure on $\mathcal{P}_{j(\nu)}(j(\nu)^{+\alpha})$.
Since $j(\nu) > \nu^{+o(y) + 1}$, we know that the set
$A^*(\alpha) := \{z \in j(A_0)(\alpha) \mid \nu^{+o(y) + 1} <
(z \cap j(\nu))\}$ is still a measure-one set. For $\alpha \in
\dom(j(A_0)) \cap j(\eta)$, set $A^*(\alpha) = j(A_0)(\alpha)$. Then $\hat{p}_0
= (j(a_0), A^*)$ is a direct extension of $j(p_0)$ in $j(\bb{P}_{U_y})$.

If there is a limit ordinal in $\dom(j(a_0))$, then let $\alpha_0^*$ be the
largest such limit ordinal, and let $\nu_0 = a_0(\alpha_0^*) \cap j(\nu)$.
Note that, letting $x = j^{-1}(a_0(\alpha_0^*))$, we have $\nu_0 =
\nu_x$ and $\nu_x^{+\bar{o}(x)} < \nu < \mu$. If there is no limit ordinal in
$\dom(j(a_0))$, then let $\nu_0 = 0$.
If there is a limit ordinal in the interval $(j(\eta), j(o(y)))$, then let
$\alpha_1^*$ be the least such ordinal, and let $\hat{p}_1$ be an extension
of $\hat{p}_0$ with $\alpha^* \in \dom(a^{\hat{p}_1})$. Notice that, in this
case, letting $\nu_1 = a^{\hat{p}_1}(\alpha_1^*) \cap j(\nu)$, our construction
of $A^*$ implies that $\mu < \nu_1$. If there is no such limit ordinal, then let
$\hat{p}_1 = \hat{p}_0$ and $\nu_1 = j(\nu)$.

We are now in the setting of Lemma \ref{club_subset}
applied to $j(\bb{P}_{\vec{U}_y})$, $\delta$, $\nu_0 < \mu < \nu_1$,
and $\hat{p}_1$. We can therefore
find a direct extension $\hat{p}_2$ of $\hat{p}_1$
in $j(\bb{P}_{\vec{U}})$ and a club $D \subseteq \mu$ in $M$ such that
$\hat{p}_2$ forces $D \subseteq \dot{C}$. Let
$E = \{ \gamma < \nu^{+o(y)+1} \mid j(\gamma) \in
D \}$.  It is straightforward to see that $E$ is ${<}\nu$-club in
$\nu^{+o(y)+1}$.  Let $\gamma^*$ be the
$\nu^{+o(y)}$-th element in an increasing enumeration of $E$.  We can
assume that there is an index $\hat{\gamma} < \nu^{+o(y)+1}$
such that $\hat{p}_2$ forces that
$\dot{C} \cap j(\gamma^*)$ is enumerated as
$j(\dot{z})_{\gamma'}$ for some $\gamma' < j(\hat{\gamma})$.

Recall that $\cf(o(y)) < \nu$.  Now if
$x \subseteq E \cap \gamma^*$ has order type at most $\cf(o(y))$, then
there is a condition $p_x \leq p_0$ in $\bb{P}_{\vec{U}_y}$ which
forces that $x \subseteq \dot{z}_\gamma$ for some
$\gamma < \hat{\gamma}$.  To see this, notice that $j$ of this statement is
witnessed by $\hat{p}_2$. Also note that the number of such $x$ is
$|E \cap \gamma^*|^{\cf(o(y))} = (\nu^{+o(y)})^{\cf(o(y))} =
\nu^{+o(y) + 1}$.

By the $\nu^{+o(y) + 1}$-chain condition of $\P_{\vec{U}_y}$,
we can find a condition which forces
that, for $\nu^{+o(y)+1}$ many $x$, $p_x$ is in the generic filter.
This is impossible, since each $\dot{z}_\gamma$ is forced to have order type
less than $\nu$ and hence, in the extension, where $\nu$ remains a strong
limit cardinal, we have $\vert
\bigcup_{\gamma<\hat{\gamma}} \mathcal{P}(\dot{z}_\gamma) \vert \leq \nu \cdot
\nu^{+o(y)} = \nu^{+o(y)}$.
\end{proof}

Next, we show that weak square holds at points taken from $X_\beta$ where
$\cf(\beta) \geq \kappa$. We first need the definition of a partial square
sequence.

\begin{definition} Let $\lambda< \delta$ be regular cardinals, and let $S \subseteq
\delta \cap \cof(\lambda)$.  We say that $S$ \emph{carries a partial square
sequence} if there is a sequence $\langle C_\gamma \mid \gamma \in S \rangle$
such that:
\begin{enumerate}
\item for all $\gamma \in S$, $C_\gamma$ is club in $\gamma$ and
$\otp(C_\gamma)=\lambda$;
\item for all $\gamma < \gamma^*$ from $S$, if $\beta$ is a limit point of
$C_\gamma$ and $C_{\gamma^*}$, then $C_\gamma \cap \beta = C_{\gamma^*} \cap
\beta$.
\end{enumerate}
\end{definition}

Next, we need a theorem of D\u{z}amonja and Shelah \cite{DS}.

\begin{theorem} \label{ds_theorem} Suppose that $\lambda$ is a regular cardinal and
	$\mu > \lambda$ is
singular. If $\cf([\mu]^{\leq \lambda}, \subseteq) = \mu$, then $\mu^+ \cap
\cof(\lambda)$ is the union of $\mu$-many sets, each of which carries a
partial square sequence.  \end{theorem}

\begin{lemma} \label{wk_square_lemma}
  Suppose that $\beta \in \theta \cap \cof({\geq}\kappa)$ and
  $p \in \bb{P}$ is a condition such that $a^p(\beta) = y$ for some $y$.  Then
	$p \Vdash \square_{\kappa_y}^*$.
\end{lemma}

\begin{proof} For each regular $\lambda < \kappa_y$, we have that
$(\kappa_y^{+o(y)})^{\leq \lambda} =
\kappa_y^{+o(y)}$ using the supercompactness of $\kappa_y$, and
hence $\cf([\kappa_y^{+o(y)}]^{\leq\lambda}, \subseteq ) =
\kappa_y^{+o(y)}$. Therefore, by Theorem \ref{ds_theorem}, we can
write $\kappa_y^{+o(y)+1} \cap \cof(\lambda)$ as the
union of $\kappa_y^{+o(y)}$-many sets which have partial squares.  We
call these partial square sequences $\vec{C}^{\lambda,i}$ for
$i<\kappa_y^{+o(y)}$.

Now let $G$ be $\bb{P}$-generic over $V$ with $p \in G$. By Lemma \ref{cofomega},
in $V[G]$, we have that each $V$-regular cardinal in the
interval $[\kappa_y,\kappa_y^{+o(y)}]$ changes its cofinality to
$\omega$ and $\kappa_y^{+o(y) + 1}$ becomes the successor of $\kappa_y$.
So, in $V[G]$, we can write $\kappa_y^{+o(y)+1}$ as
the disjoint union of $(\kappa_y^{+o(y)+1} \cap \cof({<}\kappa_y))^V$,
which we call $T_0$, and a set $T_1$ of ordinals of countable cofinality.

We define a weak square sequence as follows.  For $\gamma \in T_0$, we let
$\mathcal{C}_\gamma = \{ C_{\gamma'}^{\lambda,i} \cap \gamma \mid \lambda <
\kappa_y$, $i < \kappa_y^{+o(y)}$ and $\gamma$ is a limit point of
$C_{\gamma'}^{\lambda,i} \}$.  For $\gamma \in T_1$, we let $\mathcal{C}_\gamma =
\{C\}$ where $C$ is some cofinal $\omega$-sequence in $\gamma$.

The coherence is obvious, so we just have to check that each
$\mathcal{C}_\gamma$ is not too large.  Suppose that there is $\gamma$ such that
$\vert \mathcal{C}_\gamma \vert \geq \kappa_y^{+o(y)+1}$.  Then, by
the pigeonhole principle, we can find two elements $C$ and $C'$ on which the
indices $\lambda$ and $i$ are the same. But then we have that $C=C'$ by the
coherence of the partial square sequence with indices $\lambda$ and $i$,
which is a contradiction. \end{proof}

We are now ready to complete the proof of Theorem \ref{thm1}. Let $G$
be $\bb{P}$-generic over $V$. Our final model will be a further forcing extension
of $V[G]$. In $V[G]$, let $S$ be the set of $\nu \in \lim(C_G)$ such that
$\square^*_\nu$ holds. By Lemmas \ref{cofomega}, \ref{app_lemma}, and
\ref{wk_square_lemma}, and the fact that $\square^*_\nu$ implies
$\nu^+ \in I[\nu^+]$, we know that $S
\subseteq \kappa \cap \cof(\omega)$ and is precisely the set of
$\nu < \kappa$ such that, for some $p \in G$ and some limit ordinal
$\alpha \in a^p$ with $\cf^V(\alpha) \geq \kappa$, we have $\kappa_{a^p(\alpha)}
= \nu$. By genericity, $S$ is stationary in $\kappa$. However, we claim that $S$ can be
made non-stationary in a cofinality-preserving forcing extension of $V[G]$.

\begin{lemma} \label{non_reflecting_lemma}
  In $V[G]$, suppose that $\delta \in \lim(C_G) \cap \cof({>}\omega)$. Then
	$S \cap \delta$ is non-stationary in $\delta$.
\end{lemma}

\begin{proof}
  Fix $p = (a,A) \in G$ such that, for some $\beta \in (\theta \cap \cof({<}\kappa))^V$,
  $a(\beta) = y$, where $\kappa_y = \delta$. Work in $V$, letting $\dot{S}$ be a
  canonical $\P_{\vec{U}}$-name for $S$. We will find $q \leq p$ such that
  $q \Vdash ``\dot{S} \cap \delta$ is non-stationary$."$

  Let $\mu = \cf(\beta)$. Since $\mu < \kappa$ and $y \in X_\beta$, we have that
  $\mu < \kappa_y$ and $Z^\beta_y$ is ${<}\kappa_y$-closed and unbounded in $\beta$.
  Find $D \subseteq Z^\beta_y$ such that:
  \begin{itemize}
    \item $D$ is club in $\beta$;
    \item $\otp(D) = \mu$;
    \item $\min(D) > \max(\dom(a) \cap \beta)$.
  \end{itemize}

  For each $\alpha \in Z^\beta_y \setminus \min(D)$, let $Y_\alpha = \{x \in
    \power_\delta(y \cap \kappa^{+\alpha}) \mid D \cap \alpha \subseteq
  Z^\alpha_x\}$, and note that $Y_\alpha \in u^\beta_\alpha(y)$. Define
  $q = (a, B) \leq p$ by letting $B(\alpha) = A(\alpha) \cap Y_\alpha$ for
  $\alpha \in Z^\beta_y \setminus \min(D)$ and $B(\alpha) = A(\alpha)$ for all
  other values of $\alpha$. Now $q \Vdash ``D \subseteq \dot{C}_G."$ Let $\dot{E}$
  be a $\P_{\vec{U}}/q$-name for $\{\kappa_x \mid$ for some $r \in G$ and $\alpha
  \in D$, $r(\alpha) = x\}$. $q \Vdash ``\dot{E}$ is club in $\delta"$ and, since
  $\lim(D) \subseteq \cof({<}\kappa)$, Lemma \ref{app_lemma} implies that
  $q \Vdash ``\lim(\dot{E}) \cap \dot{S} = \emptyset."$.
\end{proof}

Recall that a stationary subset $T$ of a regular, uncountable cardinal $\lambda$
is \emph{fat} if for every club $D \subseteq \lambda$ and every ordinal
$\eta < \lambda$, $D \cap T$ contains a closed subset of order type $\eta$.
We claim that $\kappa \setminus S$ is a fat stationary subset of $\kappa$ in
$V[G]$. To see this, fix a club $D \subseteq \kappa$ and an infinite
ordinal $\eta < \kappa$. Since $\kappa$ is strongly inaccessible, we can find
$\delta \in \lim(C_G \cap D)$ with $\cf(\delta) > \eta$. By
Lemma \ref{non_reflecting_lemma}, $S \cap \delta$ is non-stationary in $\delta$,
so we can find a club $E \subseteq \delta$ that is disjoint from $S$.
But then $D \cap E$ is a closed subset of $D \cap (\kappa \setminus S)$.
Thus, $\kappa \setminus S$ is a fat stationary subset of $\kappa$.

Let $\bb{Q}$ be the forcing notion whose conditions are closed, bounded subsets of
$\kappa$ disjoint from $S$, ordered by end-extension. $\bb{Q}$ adds a club
in $\kappa$ disjoint from $S$ and, by a result of Abraham and Shelah
\cite[Theorem 1]{Abraham-Shelah} and the fact that $\kappa \setminus S$
is fat, $\bb{Q}$ is $\kappa$-distributive. Thus, if $H$ is
$\bb{Q}$-generic over $V[G]$, $D$ is the generic club added by $Q$, and
$E = D \cap C_G$, then $E$ witnesses that $V[G*H]$ satisfies the conclusion of
Theorem \ref{thm1}. Moreover, if we let $N = (V[G*H])_\kappa = V_\kappa^{V[G*H]}$,
then $N$ is a model of GB (G\"{o}del-Bernays) with a class club $E$ through its
cardinals such that, for every $\nu \in E$, $\nu$ is a singular cardinal, SCH fails
at $\nu$, and $\square^*_\nu$ fails.

\section{Conclusion} \label{conclusion}

In a forthcoming paper of the third author \cite{ungernew}, a model is
constructed in which $\aleph_{\omega^2}$ is strong limit and weak square fails
for all cardinals in the interval $[\aleph_1,\aleph_{\omega^2+2}]$.  In
particular, it is shown that one can put collapses between the Prikry points of
the Gitik-Sharon \cite{gitiksharon} construction which will make $\kappa$ into
$\aleph_{\omega^2}$ and enforce the failure of weak square below
$\aleph_{\omega^2}$.

It is reasonable to believe that this construction could be combined with the
forcing from Theorem \ref{thm1}, but we are left with the unsatisfactory result
that weak square will hold at some successors of singulars in the extension.  To
make this precise, we formulate a question which seems to capture the limit of a
naive combination of the two techniques.

\begin{question} Suppose that $\kappa$ is a singular cardinal of cofinality
$\omega$ such that $\square^*_\lambda$ fails for all $\lambda \in
[\aleph_1,\kappa)$ and $\vert \{ \lambda <\kappa \mid \lambda$ is singular
strong limit $\}\vert =\kappa$.  Is there a $\square_\kappa^*$-sequence?
\end{question}

We also ask two other natural questions.

\begin{question} Is there a version of Theorem \ref{thm1} in which the failure of $\square_\nu^*$ is replaced with the tree property at $\nu^+$?
\end{question}

\begin{question} Let $C_G \subset \kappa$ be a generic Radin club added by the poset $\po$ defined in Section \ref{mainposet}. Does $\square_{\nu,\omega}$ fail at every ordinal $\nu \in C_G$?
\end{question}

\bibliography{diagonal_radin}
\bibliographystyle{amsplain}

\end{document}